\pdfoutput=1
\documentclass{lmcs}
\pdfoutput=1

\usepackage{lastpage}
\lmcsdoi{22}{2}{30}
\lmcsheading{}{\pageref{LastPage}}{}{}%
{Jan.~06,~2026}{Jun.~17,~2026}{}


\theoremstyle{definition}\newtheorem{con}[thm]{Construction}

\usepackage[T1]{fontenc}
\usepackage{stmaryrd}
\usepackage{hyperref}
\usepackage{tikz-cd}

\usepackage{local-tikz}
\usepackage{local-abbrv}
\usepackage{local-jmsdelim}
\usepackage{jms-sections}
\usepackage{string-diagrams}
\usepackage[utf8]{inputenc}

\NewDocumentCommand\PsCoalg{m}{\text{\rm Ps-}#1\text{\rm -Coalg}}
\NewDocumentCommand\Gray{}{\mathbf{Gray}}
\NewDocumentCommand\Yo{m}{\mathbf{y}_{\!#1}}
\NewDocumentCommand\Op{}{\Sup{\textit{op}}}
\NewDocumentCommand\Co{}{\Sup{\textit{co}}}
\NewDocumentCommand\Sum{m}{\Sigma\Sub{#1}}
\NewDocumentCommand\OplSum{m}{\Sigma\Sup{\textit{opl}}\Sub{#1}}
\NewDocumentCommand\OplDiag{m}{\Delta\Sup{\textit{opl}}\Sub{#1}}
\NewDocumentCommand\LaxSum{m}{\Sigma\Sup{\textit{lax}}\Sub{#1}}
\NewDocumentCommand\LaxDiag{m}{\Delta\Sup{\textit{lax}}\Sub{#1}}

\NewDocumentCommand\CATFib{}{\CAT\Co_{\mathit{fib}}}
\NewDocumentCommand\FIBFib{}{\mathbf{FIB}_{\mathit{fib}}}

\NewDocumentCommand\Lift{mmm}{#2^* #3}
\NewDocumentCommand\LiftArr{mmm}{#2^\dagger #3}

\NewDocumentCommand\Diag{m}{\Delta\Sub{#1}}

\NewDocumentCommand\CAT{}{\mathbf{Cat}}
\NewDocumentCommand\SET{}{\mathbf{Set}}
\NewDocumentCommand\PtSET{}{\mathbf{1}/\mathbf{Set}}

\NewDocumentCommand\UU{}{\mathcal{U}}
\NewDocumentCommand\sFIB{m}{\mathbf{sFib}_{#1}}
\NewDocumentCommand\FIB{m}{\mathbf{Fib}_{#1}}
\NewDocumentCommand\DFIB{m}{\mathbf{DFib}_{#1}}

\begin{document}

\title[Hofmann--Streicher lifting of fibred categories]{Hofmann--Streicher lifting of fibred categories\\[4pt] \small\normalfont (Extended Version)}
\titlecomment{{\lsuper*}This is an extended version of \emph{Hofmann--Streicher Lifting of Fibred Categories} in LICS 2025, which was dedicated to the memory of Thomas Streicher (1958--2025).}
\thanks{This work was funded by the United States Air Force Office of Scientific Research under grant FA9550-23-1-0728 (\emph{New Spaces for Denotational Semantics}; Dr Tristan Nguyen, Program Manager). Views and opinions expressed are however those of the authors only and do not necessarily reflect those of AFOSR}
\author[A.~Slattery]{Andrew Slattery\lmcsorcid{0009-0008-4933-821X}}
\author[J.~Sterling]{Jonathan Sterling\lmcsorcid{0000-0002-0585-5564}}

\address{Computer Laboratory, University of Cambridge}
\email{aws46@cam.ac.uk, js2878@cl.cam.ac.uk}

\begin{abstract}
  In 1997, Hofmann and Streicher introduced an explicit construction to lift a Grothendieck universe from the category of sets into the category of set-valued presheaves on a small category. More recently, Awodey presented an elegant functorial analysis of this construction in terms of the \emph{categorical nerve}, the right adjoint to the functor that takes a presheaf to its category of elements; in particular, the categorical nerve's functorial action on the universal small discrete fibration gives the generic family of the universe's Hofmann--Streicher lifting.
  Inspired by Awodey's analysis, we define a relative version of Hofmann--Streicher lifting in terms of the right pseudo-adjoint to the 2-functor given by postcomposition with a fibration. Finally, we construct a new 2-bifibration of fibrations in which the opcartesian and cartesian lifts arise from these pseudo-adjunctions.
\end{abstract}

\maketitle

\tableofcontents

\begin{xsect}{Introduction}
  Universes were originally introduced to axiomatic set theory by Zermelo~\cite{zermelo:1930} in 1930, and reinvented in the early 1960s by the Grothendieck School as a technical device for controlling the size of categories~\cite{sga:4}. In the subsequent decade, several important developments occurred in short order:
  \begin{itemize}
    \item In 1971, Per Martin-L\"of introduced his own version of universes in the context of his intuitionistic type theory~\cite{martin-lof:1971}, which Girard noticed was overly strong~\cite{girard:1972}. This led Martin-L\"of to codify his now-standard presentation of (predicative) universes in type theory~\cite{martin-lof:itt:1975,martin-lof:1979,martin-lof:1984}.
    \item Around the same time, Bénabou developed a categorical axiomatisation of universes~\cite{benabou:1973} that makes sense for any elementary topos, subsuming both Zermelo and Grothendieck's set theoretic universes \emph{and} Martin-L\"of's type theoretic universes.
  \end{itemize}

\noindent
  The type theoretic and category theoretic viewpoints on universes continued to develop in parallel. By the mid 1990s,  Martin-L\"of type theory had acquired several credible notions of \emph{categorical model}~\cite{cartmell:1978,cartmell:1986,jacobs:1993,dybjer:1996}; this made it possible to consider natural questions, such as whether models of Martin-L\"of type theory are stable under standard categorical constructions like presheaves. This question was answered in the affirmative by the late Martin Hofmann and Thomas Streicher in 1997, and the critical ingredient was what we now refer to as the \emph{Hofmann--Streicher universe lifting} construction~\cite{hofmann-streicher:1997}, which we recall below as reformulated by Awodey~\cite{awodey:2024:universes}:\footnote{Hofmann and Streicher~\cite{hofmann-streicher:1997} originally presented the generic family as a presheaf on the category of elements of $\UU_{\hat{C}}$, recalling the equivalence $\hat{C}/X\simeq \widehat{\int_CX}$. Equivalently viewing the universe as an arrow, as Awodey does, makes the categorical analysis considerably more direct.}

  \begin{quote}
    \itshape
    Given a universe $p\colon\tilde\UU\to\UU$ in $\SET$ and a $\UU$-small category $C$, there is a universe $p_{\hat{C}}\colon\tilde\UU_{\hat{C}}\to \UU_{\hat{C}}$ in $\hat{C}$ defined as follows:
    \begin{align*}
      \tilde\UU_{\hat{C}}\prn{c} &= \CAT\prn{\prn{C/c}\Op,\tilde\UU}
      \\
      \UU_{\hat{C}}\prn{c} &= \CAT\prn{\prn{C/c}\Op,\UU}
      \\
      p_{\hat{C}}^c\prn{x} &= \prn{f\colon d\to c} \mapsto p\prn{x\prn{f}}
    \end{align*}
  \end{quote}

  Hofmann and Streicher's lifting construction does not depend on the axioms of universes, but it does preserve them. In fact, the construction itself requires only a full internal subcategory of $\SET$ presented by a generic family.

  \begin{xsect}{Hofmann--Streicher lifting and universal properties}
    Although universes in the sense of Bénabou~\cite{benabou:1973}, Streicher~\cite{streicher:2005}, and Martin-L\"of~\cite{martin-lof:itt:1975} are not defined by means of a universal property, Hofmann--Streicher lifting  does preserve the universal properties of more restrictive kinds of universes, including Lawvere's subobject classifiers. Indeed, when applied to the full internal subcategory of \emph{propositions} in $\SET$ which is presented by the subobject classifier $\brc{1}\hookrightarrow\mathbf{2}$, the Hofmann--Streicher lifting construction yields the subobject classifier of any presheaf topos (see Awodey~\cite{awodey:2024:universes}); indeed, we have:
    \[
      \Omega_{\hat{C}}\prn{c} = \CAT\prn{\prn{C/c}\Op,\mathbf{2}} = \mathrm{Sieve}_C\prn{c}
    \]

    \begin{rem}
      There is a subtlety here, namely that a category $C$ is never $\mathbf{2}$-small unless it is either the initial or terminal category; nonetheless, the Hofmann--Streicher lifting does exist and is the subobject classifier of $\hat{C}$. In fact, the assumption that $C$ is $\UU$-small was needed by Hofmann and Streicher only to close the lifted universe under type theoretic connectives such as function spaces; closure of $\Omega_{\hat{C}}$ under logical connectives follows by different means.
    \end{rem}

    \noindent
    The universal property of the subobject classifier is 1-dimensional and can therefore be phrased in the language of 1-toposes. The corresponding notion in a 2-dimensional topos would then classify (internal) \emph{discrete fibrations} rather than subobjects, as described by Weber~\cite{weber:2007}. In the 2-topos $\CAT$ of categories, the classifying discrete opfibration is of course $\partial_1\colon \mathbf{1}/\SET\to \SET$.

    Although Weber~\cite{weber:2007} seems to have been unaware of Hofmann and Streicher's universe lifting construction, he appears to invent it independently. He applies it to the discrete opfibration classifier $\mathbf{1}/\SET\to\SET$ of the 2-topos $\CAT$ to obtain the discrete opfibration classifier of the (strict) functor 2-topos $\CAT\prn{\hat{C}}$ consisting of strict functors $C\Op\to\CAT$, natural transformations between them, and modifications between those. Of course, $\CAT\prn{\hat{C}}$ is 2-equivalent to the 2-category of \emph{split} fibrations, with functors strictly preserving the splittings, and natural transformations between these.
  \end{xsect}

  \begin{xsect}{Functoriality of Hofmann--Streicher lifting}
    Awodey~\cite{awodey:2024:universes} recently showed that Hofmann and Streicher's original universe lifting construction is part of a 1-functor $\nu_C\colon \mathbf{cat}_{s}\to\hat{C}$ that he calls the \emph{categorical nerve}, where we have written $\mathbf{cat}_{s}$ for the 1-category of small categories and functors taken up to equality.

    Two of Awodey's crucial insights are as follows:
    \begin{enumerate}
      \item The categorical nerve $\nu_C\colon\mathbf{cat}_s\to\hat{C}$ emerges as the right adjoint to the category of elements functor $\int_C\colon\hat{C}\to\mathbf{cat}_s$. In particular, Awodey defines:
      \[
        \nu_C\prn{D}\prn{c} = \mathbf{cat}_s\prn{C/c,D}
      \]
      \item The Hofmann--Streicher lifting of a universe $\tilde\UU\to\UU$ in $\SET$ to $\hat{C}$ is obtained by applying the categorical nerve to the arrow $\tilde\UU\Op\to\UU\Op$ in $\mathbf{cat}_s$ which is the classifying discrete \emph{fibration} in $\mathbf{cat}$:
      \begin{align*}
        \nu_C\prn{\UU\Op}\prn{c}
        &= \mathbf{cat}_s\prn{C/c,\UU\Op}
        \\
        &= \mathbf{cat}_s\prn{\prn{C/c}\Op,\UU}
        \\
        &= \UU_{\hat{C}}\prn{c}
      \end{align*}
    \end{enumerate}
    \noindent
    In hindsight, it can be seen that something very similar is happening in Weber's construction~\cite{weber:2007} of the classifying discrete opfibration of $\CAT\prn{\hat{C}}$. In particular, Weber considers the following 2-adjunction in which $E_C\prn{X} = \prn{\DelimMin{1}\int_CX\Op}\Op$ and $\mathrm{Sp}_C\prn{D}\prn{c} = \CAT\prn{\prn{C/c}\Op,D}$:
    \[
      \begin{tikzcd}[column sep=huge,cramped]
        \CAT
          \arrow[r,yshift=-1.5ex,swap,"\mathrm{Sp}_C"]
          \arrow[r,phantom,"\bot"]
        &
        {\CAT\prn{\hat{C}}}
          \arrow[l,yshift=1.5ex,swap,"E_C"]
      \end{tikzcd}
    \]

    To properly compare Awodey's and Weber's constructions, we first factor Weber's construction through the duality involutions to reveal a 2-dimensional version of Awodey's categorical nerve, which we write $\int_C\dashv\nu_C\colon \CAT\to\CAT\prn{\hat{C}}$:
    \[
      \begin{tikzcd}[column sep=large,cramped]
        \CAT
          \arrow[r,swap,yshift=-1.5ex,"\mathit{op}"]
          \arrow[r,phantom,"\bot"]
        &
        \CAT\Co
          \arrow[r,swap,yshift=-1.5ex,"\nu_C\Co"]
          \arrow[r,phantom,"\bot"]
          \arrow[l,yshift=1.5ex,swap,"\mathit{op}"]
        &
        \CAT\prn{\hat{C}}\Co
          \arrow[r,swap,yshift=-1.5ex,"\mathit{op}"]
          \arrow[r,phantom,"\bot"]
          \arrow[l,yshift=1.5ex,swap,"\int_C\Co"]
        &
        {\CAT\prn{\hat{C}}}
          \arrow[l,yshift=1.5ex,swap,"\mathit{op}"]
      \end{tikzcd}
    \]

    The Grothendieck construction $\int_C\colon \CAT\prn{\hat{C}}\to \CAT$ takes a presheaf of categories to its oplax colimit; its conjugation by duality involutions $\mathit{op}\circ\int_C\Co\circ\mathit{op}$ therefore takes a presheaf of categories to its \emph{lax} colimit. Therefore, the 2-dimensional version $\nu_C\colon\CAT\to\CAT\prn{\hat{C}}$ is some kind of ``oplax base change 2-functor'' and therefore the Hofmann--Streicher lifting 2-functor is some kind of ``lax base change 2-functor''.
  \end{xsect}

  \begin{xsect}{Our contributions}
    We expose a further generalisation of Hofmann--Streicher lifting. In particular, we describe a \emph{relative} version of Hofmann--Streicher lifting defined in terms of a right \emph{pseudo}-adjoint to the postcomposition functor $\prn{p\circ-}\colon \FIB{A}\to\FIB{B}$ for a given fibration $p\colon A\rightarrowtriangle B$, where $\FIB{C}$ is the 2-category of fibred categories over $C$, fibred functors over $C$, and fibred natural transformations between them.

	    We therefore differ from the prior work in main two ways: after replacing $\CAT\prn{\hat{C}}$ with $\FIB{C}$, we are then able to relativise the total category 2-functor $\int_A\colon\FIB{A}\to\CAT$ to the more general postcomposition 2-functor $\prn{p\circ -}\colon \FIB{A}\to \FIB{B}$ induced by any fibration $p\colon A\rightarrowtriangle B$. Indeed, under the identification $\CAT\cong\FIB{\mathbf{1}}$, the total category 2-functor is precisely the postcomposition with the universal functor $A\to\mathbf{1}$. The fibrational relativisation of Hofmann--Streicher lifting makes it possible for the first time to consider \emph{iterated} lifting of universes and other structures.

    \begin{xsect}[sec:why-not-split]{From split fibrations to fibrations}
      The most direct generalisation of Weber's construction would have yielded a 2-functor $\CAT\prn{\hat{B}}\to\CAT\prn{\hat{A}}$ but it is ultimately more useful to work with fibrations.  Indeed, $\CAT\prn{\hat{C}}$ is 2-equivalent to the 2-category $\mathbf{sFib}_C$ of split fibrations over $C$, split fibred functors over $C$, and fibred natural transformations between them. Unfortunately, the evident forgetful 2-functor $U\colon \sFIB{C}\to \FIB{C}$ is not even a bi-equivalence, although Giraud exhibits both left and right 2-adjoints~\cite{giraud:1971}.
      Therefore, although it does not matter in the end whether one asks for fibrations or split fibrations (since every fibration is equivalent to a split fibration), it is overly restrictive in practice to ask that these splittings be preserved on the nose by morphisms. This is a special case of a familiar pattern in higher categorical algebra, where one prefers to consider \emph{weak} homomorphisms regardless of whether the algebras are strict or weak.
    \end{xsect}

    \begin{xsect}[sec:why-relativise]{Relativising along a fibration \texorpdfstring{$p\colon A\rightarrowtriangle B$}{p:A->B}}
      As it stands, the Hofmann--Streicher construction cannot be iterated. For example, we may take a universe $\UU\in\SET$ of sets to a universe $\UU_{\hat{C}}\in\hat{C}$ of presheaves; but if $D$ is an internal category in $\hat{C}$, we do not have any way to lift $\UU_{\hat{C}}$ to a universe of internal presheaves on $D$ over $\hat{C}$.
      To solve this problem, we \emph{relativise} the Hofmann--Streicher construction so that it goes not from $\CAT$ to $\FIB{A}$ but rather from $\FIB{B}$ to $\FIB{A}$, parameterised in a fibration $p\colon A\rightarrowtriangle B$. When $B=\mathbf{1}$, we recover the usual Hofmann--Streicher lifting.

      To see how relativisation helps with iteration, we return to our example of a universe $\UU$ of sets, which we can view as a category of $\UU$-small sets and functions between them. Given a small category $C$, we can lift $\UU$ to a fibred category $\UU_C$ over $C$ using the ordinary Hofmann--Streicher lifting. Next, an internal category $D\in\hat{C}$ can be viewed instead as a fibration $p_D\colon D\rightarrowtriangle C$; our \emph{relative} Hofmann--Streicher lifting of $\UU_C\in\FIB{C}$ along the fibration $p_D\colon D\rightarrowtriangle C$ produces, then, a fibration $\UU_D\in \FIB{D}$ which is the desired universe of internal presheaves on $D$ over $\hat{C}$. Here we are using the standard observation that for a fibration $p\colon D\rightarrowtriangle C$, an additional fibration $q\colon E\rightarrowtriangle D$ can equally well be viewed as a fibration in $\CAT$ over $D$ \emph{or} as an internal fibration in $\FIB{C}$ over $p$.

      If ordinary Hofmann--Streicher lifting is the needful tool for defining presheaf models of Martin-L\"of type theory, then the relative Hofmann--Streicher lifting is precisely what is needed to define \emph{internal} presheaf models of Martin-L\"of type theory within another presheaf model. In fact, this arises more often than one might at first guess: for example, Bizjak~\etal~\cite{bgcmb:2016} describe what \emph{ought} to amount to the standard presheaf model of cubical type theory~\cite{cchm:2017,abcfhl:2021} internal to the presheaf model of guarded dependent type theory~\cite{bmss:2011,bizjak-mogelberg:2020}, but lacking the tools developed here, it was necessary to simultaneously describe the guarded and cubical aspects.
    \end{xsect}

    \begin{xsect}{Relative Hofmann--Streicher lifting as a pseudo-adjoint}
      Given a fibration $p\colon A\rightarrowtriangle B$, our \emph{relative} Hofmann--Streicher lifting takes the form of a 2-functor $\FIB{B}\to \FIB{A}$, which is derived from the right \emph{pseudo}-adjoint to the 2-functor $\prn{p\circ-}\colon \FIB{A}\to \FIB{B}$ that sends a fibration $q\colon E\rightarrowtriangle A$ to the composite fibration $p\circ q\colon E\rightarrowtriangle B$.
    \end{xsect}

    \begin{xsect}{A new 2-dimensional bifibration of fibred categories}
      It is well-known that the base projection 2-functor $\FIB{}\rightarrowtriangle\CAT$ is both a 2-fibration and a 2-opfibration, \emph{i.e.}\ a 2-bifibration: in essence, this is because pullback of fibrations has a left adjoint. In the original version of this paper~\cite{slattery-sterling:2025} presented at LICS~2025, we raised the question of what 2-dimensional bifibration our pseudoadjunctions $\OplSum{p}\dashv\OplDiag{p}\colon\FIB{A}\to\FIB{B}$ induce, noting that at the very least the base 2-category would need to have only \emph{fibrations} as 1-cells; beyond this, it was unclear at the time what the 2-cells ought to be.

      In this extended edition, we resolve our old question by constructing a 2-dimensional opfibration $\partial_1\colon \FIBFib\to \CATFib$ where $\CATFib$ is the 2-category of categories, fibrations between them, and natural transformations \emph{going the opposite direction}. The fibre 2-category over $B\in \CATFib$ has fibred categories over $B$ for 0-cells, and \emph{displayed} rather than fibred functors between the fibred categories as 1-cells.
    \end{xsect}
  \end{xsect}
\end{xsect}

\begin{xsect}{Technical overview}

  In this section, we provide a high-level overview of the mathematical content of the paper; this can be read by someone who wishes to understand the overall architecture of the generalised Hofmann--Streicher lifting and its applications, without necessarily grappling with all the details. For the full technical development, see \S~\ref{sec:full-development}--\ref{sec:relative-hs-lifting}.

  \begin{xsect}{Displayed and fibred categories}
    We presuppose some knowledge of fibred category theory. On top of this, we will at times employ a convention that assists with remembering when a particular fibration is to be thought of a special kind of parameterised category \emph{vs.}\ as a mere functor satisfying some properties.
    \begin{enumerate}
      \item We will refer to objects of $\CAT/B$ and $\FIB{B}$ as ``displayed'' and ``fibred'' categories to emphasise their role as objects of a 2-category of category-like things.
      \item Given a displayed category $E\in \CAT/B$, we will write $B.E$ for its total category and we shall refer to the projection functor $p_{B.E}\colon B.E\to B$ as its \emph{display}.
      \item Given a fibred category $E\in\FIB{B}$, we shall refer to the projection $p_{B.E}\colon B.E\rightarrowtriangle B$ as the corresponding \emph{fibration}.
    \end{enumerate}
    \noindent
    The specifics of how to represent displayed and fibred categories do not play a role in this paper; some readers may prefer to think of them as represented by their underlying display functors, whereas other readers may prefer to think of them in the terms advocated by Ahrens and Lumsdaine~\cite{ahrens-lumsdaine:2019}. We take the latter viewpoint in our technical development, but it is not mandatory.
  \end{xsect}

  \begin{xsect}{Base change and sum of displayed categories}
    Let $p\colon A\to B$ be an arbitrary functor. There is a \emph{change of base} 2-functor $\Diag{p}\colon\CAT/B\to\CAT/A$ computed on display functors by \emph{pseudo-pullback} along $p$; the change of base 2-functor has a left pseudo-adjoint $\Sum{p}\colon \CAT/A\to\CAT/B$ computed on display functors by \emph{postcomposition} with $p$. As left pseudo-adjoint to change of base, $\Sum{p}$ deserves to be called the \emph{sum} along $p$.
  \end{xsect}

  \begin{xsect}{Base change and sum of fibred categories}
    For any functor $p\colon A\to B$, the base change 2-functor $\Diag{p}\colon \CAT/B\to \CAT/A$ restricts to a base change 2-functor $p^*\colon \FIB{B}\to\FIB{A}$ as follows:\footnote{In fact, $p^*\colon \FIB{B}\to\FIB{A}$ can be computed on fibrations in terms of strict pullbacks if you like, because a strict pullback of a fibration is always a pseudo-pullback~\cite{joyal-street:1993}.}
    \[
      \DiagramSquare{
        width = 2.5cm,
        north/style = {->,exists},
        west = U_B,
        east = U_A,
        north = p^*,
        south = \Diag{p},
        nw = \FIB{B},
        sw = \CAT/B,
        ne = \FIB{A},
        se = \CAT/A,
      }
    \]

    The base change 2-functor $p^*\colon \FIB{B}\to\FIB{A}$ always has a left pseudo-adjoint $p_!\colon \FIB{A}\to\FIB{B}$ that deserves to be called the \emph{sum (of fibrations) along $p$}, but this does \emph{not} factor through the sum of $\Sum{p}$ of displayed categories. Indeed, this can be seen immediately because if $p\colon A\to B$ is not a fibration, then the sum of the underlying displayed category of a fibred category $E$ along $p$ need not be a fibred category.

    Interpreting B\'enabou, Streicher~\cite{streicher:2021:fib} explains that when $p\colon A\rightarrowtriangle\mathbf{1}$ is the punctual map, the sum of fibrations $p_!E$ is the localisation of $\Sum{p}\prn{E}$ that inverts all cartesian arrows; put more simply, $p_!E$ is the \emph{pseudo-colimit} of the pseudofunctor corresponding to $E$ under straightening (which we will denote by $E_\bullet\colon A\Op \to \CAT$), whereas $\Sum{p}\prn{E}$ is the \emph{oplax colimit} of the same pseudofunctor.
  \end{xsect}

  \begin{xsect}{Oplax sum and base change of fibred categories}\label{sec:intro:oplax-sum-and-base-change}
    Whenever $p\colon A\rightarrowtriangle B$ is a fibration, the sum of displayed categories along $p$ restricts to a 2-functor on fibred categories that computes what we would most rightly call the ``oplax sum'' along $p$.
    \[
      \DiagramSquare{
        width = 2.5cm,
        north/style = {->,exists},
        west = U_B,
        east = U_A,
        north = \OplSum{p},
        south = \Sum{p},
        nw = \FIB{B},
        sw = \CAT/B,
        ne = \FIB{A},
        se = \CAT/A,
      }
    \]

    Streicher~\cite[Theorem~4.1]{streicher:2021:fib} points out that in the case of a fibred category $E\in\FIB{A}$, the displayed projection functor $\OplSum{p}\prn{E}\to A$ over $B$ is fibred in the sense of preserving cartesian arrows. This tells us that cartesian lifts in $\OplSum{p}\prn{E}$ are computed successively using cartesian lifts in $A$ over $B$ and in $E$ over $A$.

    Our main technical result is to show that the oplax sum $\OplSum{p}\colon \FIB{A}\to\FIB{B}$ has a \emph{right} pseudo-adjoint $\OplDiag{p}\colon\FIB{B}\to\FIB{A}$ which can be used to compute a generalisation of Hofmann and Streicher's universe construction, and to construct this pseudo-adjoint explicitly. We are forced by convention, it would seem, to refer to $\OplDiag{p}$ as the \emph{oplax base change} of fibred categories along $p$.

    The fibred Yoneda lemma gives a quick heuristic to compute $\OplDiag{p}\prn{F}$ as a split fibration for any $F\in\FIB{B}$ as follows (where we have written $\Yo{A}\colon A\hookrightarrow \FIB{A}$ for the functor that sends $a\in A$ to the discrete domain fibration $\partial_1\colon A/a\to A$): its straightening $\OplDiag{p}\prn{F}_\bullet \colon A\Op\to \CAT$ should satisfy
      \[\OplDiag{p}\prn{F}_\bullet \cong \FIB{B}\prn{\OplSum{p}\prn{\Yo{A}\prn{-}},F}\]

    Of course, one still must check that this construction really does yield a right pseudo-adjoint to $\OplSum{p}$.

  \end{xsect}

  \begin{xsect}{Lax sum and base change of fibred categories}
    By conjugating with duality involutions, we can obtain \emph{lax} versions of sum and base change for fibred categories. In particular, if we have the pseudo-adjunction $\OplSum{p}\dashv \OplDiag{p}\colon \FIB{B}\to\FIB{A}$, we obtain a corresponding pseudo-adjunction $\LaxSum{p}\dashv\LaxDiag{p}\colon\FIB{B}\to\FIB{A}$ as follows:
    \[
      \begin{tikzcd}
        & \FIB{A}\Co\arrow[r,"\prn{\OplSum{p}}\Co"]
        & \FIB{B}\Co\arrow[dr,sloped,"\mathit{op}"]
        \\
        \FIB{A}
          \arrow[ur,sloped,"\mathit{op}"]
          \arrow[rrr,dashed,bend left=12,"\LaxSum{p}"description]
          \arrow[rrr,phantom,"\bot"]
        &&&
        \FIB{B}
          \arrow[lll,dashed,bend left=12,"\LaxDiag{p}"description]
          \arrow[dl,sloped,swap,"\mathit{op}"]
        \\
        & \FIB{A}\Co\arrow[ul,sloped,swap,"\mathit{op}"]
        & \FIB{B}\Co\arrow[l,"\prn{\OplDiag{p}}\Co"]
      \end{tikzcd}
    \]
  \end{xsect}

  \begin{xsect}[sec:intro:lax-base-change]{Hofmann--Streicher lifting as lax base change}
    Let $p\colon A\rightarrowtriangle B$ be a fibration, and let $E\in\FIB{B}$ be a fibred category over $B$;	 then we can finally define the generalised Hofmann--Streicher lifting of $E$ along $p$ to be the \emph{lax base change} $\LaxDiag{p}\prn{E}\in\FIB{A}$. In order to explain why this deserves the name, we consider the case where ${!_A}\colon A\rightarrowtriangle\mathbf{1}$ is the punctual map; for any category $E\in\CAT\cong\FIB{\mathbf{1}}$, we compute the fibres of the lax base change $\LaxDiag{!_A}\prn{E}\in\FIB{A}$:
    \begin{align*}
      \LaxDiag{!_A}\prn{E}_\bullet
      &= \prn{\OplDiag{!_A}\prn{E\Op}}\Op_\bullet
      \\
      &=
      \prn{
        \OplDiag{!_A}\prn{E\Op}_\bullet
      }\Op
      \\
      &\cong
      \CAT\prn{
        \OplSum{!_A}\prn{\Yo{A}\prn{-}},
        E\Op
      }\Op
      \\
      &=
      \CAT\prn{
        A/-,
        E\Op
      }\Op
      \\
      &=
      \CAT\prn{
        \prn{A/-}\Op,
        E
      }
    \end{align*}

    Now let $\UU$ be a Grothendieck universe presented by a generic family of $\UU$-small sets $\pi_\UU\colon\tilde\UU\to\UU$; we now assume that $A$ is $\UU$-small. The generic family of $\UU$ is the image in $\SET$ under $\mathrm{ob}\colon\CAT\to\SET$ of the forgetful functor $\partial_1\colon\PtSET_\UU\to\SET_\UU$ from pointed $\UU$-small sets to $\UU$-small sets.
    Applying the functorial action of $\LaxDiag{!_A}\colon \CAT\to\FIB{A}$ to this forgetful functor, followed by that of $\mathrm{ob}\colon \FIB{A}\to\DFIB{A}\cong\hat{A}$ gives us \emph{precisely} the generic family of Hofmann and Streicher's lifted presheaf universe:
    \[
      \mathrm{ob}\prn{\CAT\prn{\prn{A/-}\Op,\partial_1}}
      \colon
      \mathrm{ob}\prn{\CAT\prn{\prn{A/-}\Op,\mathbf{1}/\SET_{\UU}}}
      \to
      \mathrm{ob}\prn{\CAT\prn{\prn{A/-}\Op,\SET_\UU}}
    \]
  \end{xsect}

\end{xsect}

\begin{xsect}[sec:abstract-proof]{Existence of oplax base change}
  In this section we use existing results of two-dimensional category theory to establish the \emph{existence} of a right pseudo-adjoint to the oplax sum $\OplSum{p}\colon \FIB{A}\to\FIB{B}$. The key result is Corollary~9.1 of Nunes~\cite{lucatelli-nunes:2016}, which we recall below. (We thank Nathanael Arkor for pointing this out to us.)

  \begin{prop}\label{prop:pseudo-adjoint-lifting}
    Let $T \colon C \to C$ and $S \colon D \to D$ be pseudo-comonads. In the commutative diagram of pseudofunctors
    \[
      \DiagramSquare{
        nw = \PsCoalg{T},
        ne = \PsCoalg{S},
        sw = C,
        se = D,
        south = F,
        north = \hat{F},
        west = U_T,
        east = U_S,
        width = 3cm,
      }
    \]
    where $U_T$ and $U_S$ are the forgetful 2-functors from the pseudo-coalgebra 2-categories to their base 2-categories, if:
    \begin{itemize}[label=$\triangleright$]
      \item $F$ has a right pseudo-adjoint, and
      \item $\text{Ps-}T\text{-Coalg}$ has descent objects (a type of 2-dimensional limit),
    \end{itemize} then $\hat F$ also has a right pseudo-adjoint.
  \end{prop}

  \begin{cor}
    For a fibration $p\colon A\rightarrowtriangle B$, the oplax sum 2-functor $\OplSum{p}\colon \FIB{A}\to\FIB{B}$ has a right pseudo-adjoint.
  \end{cor}
  \begin{proof}
    Our aim is to apply Proposition~\ref{prop:pseudo-adjoint-lifting} to the following commutative diagram which we recall from \S~\ref{sec:intro:oplax-sum-and-base-change}:
    \[
      \DiagramSquare{
        nw = \FIB{A},
        ne = \FIB{B},
        sw = \CAT/A,
        se = \CAT/B,
        south = \Sum{p},
        north = \OplSum{p},
        west = U_A,
        east = U_B,
        width = 3cm,
      }
    \]
    We first utilise Theorem 3.4 of Emmenegger~\etal~\cite{emmenegger-mesiti-rosolini-streicher:2024}, who present a lax idempotent pseudo-comonad $N_A$ on $\CAT/A$ whose 2-category of pseudo-coalgebras is equivalent to the 2-category $\FIB{A}$ of Grothendieck fibrations.

    Next, we recall that $\Sum{p}$ has the right pseudo-adjoint $\Diag{p}$, \viz the base change 2-functor. Finally, we note that $\FIB{A}$ has descent objects because it is bicategorically complete: by the Grothendieck construction, $\FIB{A}$ is biequivalent to the 2-category of pseudofunctors from $A\Op$ to $\CAT$. Hence we may apply Proposition~\ref{prop:pseudo-adjoint-lifting} to deduce that $\OplSum{p}$ has a right pseudo-adjoint.
  \end{proof}

 This proof does not immediately tell us how the right pseudo-adjoint of $\OplSum{p}$ behaves. In the next section we provide an explicit construction and show that it does indeed give the required pseudo-adjoint.
\end{xsect}

\begin{xsect}[sec:full-development]{Explicit computation of oplax base change}
  In what follows, we assume that all fibred categories are cloven, but not necessarily split; naturally, we do not require cleavings to be preserved by fibred functors. The impact of this choice depends on the foundational system in which this paper is interpreted:
  \begin{enumerate}
    \item In foundations where the axiom of choice holds, any fibred category can be equipped with a cleaving.\footnote{Technically, one would need the axiom of choice for \emph{classes} or \emph{large sets}.} Thus the forgetful 2-functor from cloven fibred categories to uncloven fibred categories would be a biequivalence.
    \item In foundations where choice may not hold (\eg a topos), the correct notion of fibred category is usually the cloven one anyway, which agrees with Chevalley's definition.
    \item In univalent foundations, a (univalent) fibred category is \emph{uniquely} cloven~\cite{ahrens-lumsdaine:2019}, so the question doesn't arise.
  \end{enumerate}

  \noindent
  We will use the following notation for the components of the given cleaving of a fibred category $E$:
  \[
    \begin{tikzpicture}[diagram]
      \SpliceDiagramSquare{
        nw = \Lift{E}{b_{01}}{e},
        ne = e,
        sw = b_0,
        se = b_1,
        south = b_{01},
        north = \LiftArr{E}{b_{01}}{e},
        nw/style = pullback,
        width = 3cm,
        height = 1.5cm,
        west/style = {|->},
        east/style = {|->},
      }
      \node[right = of ne] (B-E) {$B.E$};
      \node[right = of se] (B) {$B$};
      \draw[fibration] (B-E) to (B);
    \end{tikzpicture}
  \]

  \begin{xsect}{2-functoriality of oplax base change}
    In this section, we will define an ``oplax base change'' 2-functor $\OplDiag{p}\colon\FIB{B}\to\FIB{A}$ for any fibration $p\colon A\rightarrowtriangle B$.

    \begin{xsect}{Action of oplax base change on on 0-cells}
      \begin{con}[Oplax base change of a fibred category]\label{con:oplax-base-change}
        Let $E\in\FIB{B}$ be a fibred category over $B$.
        We first define a displayed category $\OplDiag{p}\prn{E}$ over $A$ called the \emph{oplax base change of $E$ along $p$} as follows.
        \begin{itemize}[label=$\triangleright$]
          \item A displayed object of $\OplDiag{p}\prn{E}$ over $a\in A$ is defined to be a fibred functor  $e\colon \OplSum{p}\prn{\Yo{A}\prn{a}}\to E$ over $B$. This is depicted in the surface diagram calculus~\cite{hummon:2012} as the following 1-cell in $\Gray$~\cite{gray:1974}:
          \begin{equation*}
            \begin{tikzpicture}[diagram]
              \node (e) {$e$};
              \node[below=1.25cm of e] (a) {$a$};
              \draw[{|->}] (e) to (a);
            \end{tikzpicture}
            \quad::\quad
            \begin{tikzpicture}[line width=.5pt,baseline=(current bounding box.center)]
              \begin{scope}[canvas is zx plane at y=0, yscale=-1, xscale=-1,transform shape]
                \CreateRect{4}{2.5}
                \path
                  coordinate (a) at (spath cs:north 0.25)
                  coordinate (yo) at (spath cs:north 0.5)
                  coordinate (sum-p) at (spath cs:north 0.75)
                  coordinate (E) at (spath cs:south 0.5)
                ;
                \path[spath/save global=horiz] (yo) to (E);
                \path[spath/save global=cup] (sum-p) to[in=-90,out=-90,looseness=2] coordinate[dot] (e) (a);
              \end{scope}
              \draw[spath/use=horiz,spath/use=cup,];
              \path
                node[anchor=east] at (yo.west) {\scriptsize$\Yo{A}$}
                node[anchor=east] at (a.west) {\scriptsize$a$}
                node[anchor=east] at (sum-p.west) {\scriptsize$\OplSum{p}$}
                node[anchor=west] at (E.east) {\scriptsize$E$}
                node[anchor=north west] at (e.south east) {\scriptsize$e$};
            \end{tikzpicture}
          \end{equation*}

          \item A displayed arrow from $e_0$ to $e_1$ over $a_{01}\colon a_0\to a_1$ is given by a 2-cell \[e_{01}\colon e_0\to e_1\circ \OplSum{p}\prn{\Yo{A}\prn{a_{01}}}\] in $\FIB{B}\prn{\OplSum{p}\prn{\Yo{A}\prn{a_0}},E}$; this is depicted as a 3-cell in $\Gray$ as follows:
          \[
            \DiagramSquare{
              nw = e_0,
              ne = e_1,
              north = e_{01},
              sw = a_0,
              se = a_1,
              south = a_{01},
              west/style = {|->},
              east/style = {|->},
              height = 1.5cm,
            }
            \quad::\quad
            \begin{tikzpicture}[line width=.5pt,baseline=(current bounding box.center)]
              \begin{scope}[canvas is zx plane at y=2,yscale=-1,xscale=-1,transform shape]
                \CreateRect{3}{2.5}
                \path
                  coordinate (n-a0) at (spath cs:north 0.25)
                  coordinate (n-yo) at (spath cs:north 0.5)
                  coordinate (n-sum-p) at (spath cs:north 0.75)
                  coordinate (n-E) at (spath cs:south 0.5)
                ;
                \path[spath/save global=n-horiz] (n-yo) to (n-E);
                \path[spath/save global=n-cup] (n-sum-p) to[in=-90,out=-90,looseness=3] coordinate[dot] (e0) (n-a0);
              \end{scope}

              \begin{scope}[canvas is zx plane at y=0, yscale=-1, xscale=-1,transform shape]
                \CreateRect{3}{2.5}
                \path
                  coordinate (s-a0) at (spath cs:north 0.25)
                  coordinate (s-yo) at (spath cs:north 0.5)
                  coordinate (s-sum-p) at (spath cs:north 0.75)
                  coordinate (s-E) at (spath cs:south 0.5)
                ;
                \path[spath/save global=s-horiz] (s-yo) to (s-E);
                \path[spath/save global=s-cup] (s-sum-p) to[in=-90,out=-90,looseness=3] coordinate[dot] (e1) (s-a0);
                \coordinate[dot] (a01) at (spath cs:s-cup 0.8);
              \end{scope}

              \path[spath/save=e0-e1] (e0.center) to coordinate[dot,label=right:$e_{01}$] (e01) (e1.center);
              \path[spath/save=e01-a01,bend right=10] (e01.center) to (a01.center);

              \path[spath/save=merid-a0] (n-a0) to (s-a0);
              \path[spath/save=merid-yo] (n-yo) to (s-yo);
              \path[spath/save=merid-sum-p] (n-sum-p) to (s-sum-p);
              \path[spath/save=merid-E] (n-E) to (s-E);

              \tikzset{
                spath/split at intersections with={s-cup}{e0-e1},
                spath/split at intersections with={s-cup}{e01-a01},
                spath/split at intersections with={s-horiz}{e01-a01},
                spath/split at intersections with={merid-yo}{n-cup},
                spath/split at intersections with={merid-sum-p}{n-cup},
                spath/split at intersections with={merid-sum-p}{n-horiz},
                spath/insert gaps after components={s-cup}{1.5pt},
                spath/insert gaps after components={s-horiz}{1.5pt},
                spath/insert gaps after components={merid-sum-p}{3pt},
                spath/insert gaps after components={merid-yo}{3pt}
              }

              \draw[
                spath/use=n-horiz,
                spath/use=s-horiz,
                spath/use=n-cup,
                spath/use=s-cup,
                spath/use=e0-e1,
                spath/use=e01-a01,
              ];
              \draw[
                very densely dotted,
                spath/use=merid-a0,
                spath/use=merid-yo,
                spath/use=merid-sum-p,
                spath/use=merid-E,
              ];

              \path
                node[anchor=west] at (n-E.east) {\scriptsize$E$};

              \path
                node[anchor=north] at (a01.south) {\scriptsize$a_{01}$}
                node[anchor=north west] at (e1.south east) {\scriptsize$e_1$}
                node[anchor=south west] at (e0.north east) {\scriptsize$e_0$};
            \end{tikzpicture}
          \]

          \item The displayed identity arrow $1_e\colon e\to e$ over $a\in A$ is the corresponding identity arrow $1_e\colon e\to e = e\circ \OplSum{p}\prn{\Yo{A}\prn{1_a}}$.
          \[
            \DiagramSquare{
              nw = e,
              ne = e,
              north/style = {double},
              sw = a,
              se = a,
              south/style = {double},
              west/style = {|->},
              east/style = {|->},
              height = 1.5cm,
            }
            \quad\coloneq\quad
            \begin{tikzpicture}[line width=.5pt, baseline=(current bounding box.center)]
              \begin{scope}[canvas is zx plane at y=1,yscale=-1,xscale=-1,transform shape]
                \CreateRect{3}{2.5}
                \path
                  coordinate (n-a) at (spath cs:north 0.25)
                  coordinate (n-yo) at (spath cs:north 0.5)
                  coordinate (n-sum-p) at (spath cs:north 0.75)
                  coordinate (n-E) at (spath cs:south 0.5)
                ;
                \path[spath/save global=n-horiz] (n-yo) to (n-E);
                \path[spath/save global=n-cup] (n-sum-p) to[in=-90,out=-90,looseness=3] coordinate[dot] (n-e) (n-a);
              \end{scope}
              \begin{scope}[canvas is zx plane at y=0,yscale=-1,xscale=-1,transform shape]
                \CreateRect{3}{2.5}
                \path
                  coordinate (s-a) at (spath cs:north 0.25)
                  coordinate (s-yo) at (spath cs:north 0.5)
                  coordinate (s-sum-p) at (spath cs:north 0.75)
                  coordinate (s-E) at (spath cs:south 0.5)
                ;
                \path[spath/save global=s-horiz] (s-yo) to (s-E);
                \path[spath/save global=s-cup] (s-sum-p) to[in=-90,out=-90,looseness=3] coordinate[dot] (s-e) (s-a);
              \end{scope}

              \path[spath/save=idn-e] (n-e) to (s-e);
              \path[spath/save=merid-a] (n-a) to (s-a);
              \path[spath/save=merid-yo] (n-yo) to (s-yo);
              \path[spath/save=merid-sum-p] (n-sum-p) to (s-sum-p);
              \path[spath/save=merid-E] (n-E) to (s-E);

              \tikzset{
                spath/split at intersections with={s-cup}{idn-e},
                spath/split at intersections with={merid-yo}{n-cup},
                spath/split at intersections with={merid-sum-p}{n-cup},
                spath/split at intersections with={merid-sum-p}{n-horiz},
                spath/insert gaps after components={s-cup}{1.5pt},
                spath/insert gaps after components={s-horiz}{1.5pt},
                spath/insert gaps after components={merid-sum-p}{3pt},
                spath/insert gaps after components={merid-yo}{3pt}
              }

              \draw[
                very densely dotted,
                spath/use=merid-a,
                spath/use=merid-yo,
                spath/use=merid-sum-p,
                spath/use=merid-E,
              ];

              \draw[
                spath/use=n-horiz,
                spath/use=s-horiz,
                spath/use=n-cup,
                spath/use=s-cup,
                spath/use=idn-e
              ];
              \path
                node[anchor=south west] at (n-e.north east) {\scriptsize$e$}
                node[anchor=north west] at (s-e.south east) {\scriptsize$e$}
                ;
            \end{tikzpicture}
          \]

          \item Given displayed arrows $e_{01}\colon e_0\to e_1$ and $e_{12}\colon e_1\to e_2$ over $a_{01}\colon a_0\to a_1$ and $a_{12}\colon a_1\to a_2$ respectively, the displayed composite $e_{12}\circ e_{01}\colon e_0\to e_2$ over $a_{12}\circ a_{01}$ is obtained as follows:
          \[
            \begin{tikzpicture}[diagram]
              \SpliceDiagramSquare<l/>{
                nw = e_0,
                ne = e_1,
                sw = a_0,
                se = a_1,
                south = a_{01},
                north = e_{01},
                west/style = {|->},
                east/style = {|->},
                height = 1.5cm,
              }
              \SpliceDiagramSquare<r/>{
                glue = west,
                glue target = l/,
                ne = e_2,
                se = a_2,
                north = e_{01},
                south = a_{01},
                east/style = {|->},
                height = 1.5cm,
              }
            \end{tikzpicture}
            \quad\coloneq\quad
            \begin{tikzpicture}[line width=.5pt, baseline=(current bounding box.center)]
              \begin{scope}[canvas is zx plane at y=3,yscale=-1,xscale=-1,transform shape]
                \CreateRect{3}{2.5}
                \path
                  coordinate (n-a0) at (spath cs:north 0.25)
                  coordinate (n-yo) at (spath cs:north 0.5)
                  coordinate (n-sum-p) at (spath cs:north 0.75)
                  coordinate (n-E) at (spath cs:south 0.5)
                ;
                \path[spath/save global=n-horiz] (n-yo) to (n-E);
                \path[spath/save global=n-cup] (n-sum-p) to[in=-90,out=-90,looseness=3] coordinate[dot] (e0) (n-a0);
              \end{scope}

              \begin{scope}[canvas is zx plane at y=0, yscale=-1, xscale=-1,transform shape]
                \CreateRect{3}{2.5}
                \path
                  coordinate (s-a0) at (spath cs:north 0.25)
                  coordinate (s-yo) at (spath cs:north 0.5)
                  coordinate (s-sum-p) at (spath cs:north 0.75)
                  coordinate (s-E) at (spath cs:south 0.5)
                ;
                \path[spath/save global=s-horiz] (s-yo) to (s-E);
                \path[spath/save global=s-cup] (s-sum-p) to[in=-90,out=-90,looseness=3] coordinate[dot] (e2) (s-a0);
                \path
                  coordinate[dot] (a01) at (spath cs:s-cup 0.95)
                  coordinate[dot] (a12) at (spath cs:s-cup 0.75);
              \end{scope}

              \path[spath/save=e0-e2] (e0.center) to (e2.center);

              \path
                coordinate[dot,label=right:$e_{01}$] (e01) at (spath cs:e0-e2 0.33)
                coordinate[dot,label=right:$e_{12}$] (e12) at (spath cs:e0-e2 0.66);

              \path[spath/save=e01-a01,bend right=10] (e01.center) to (a01.center);
              \path[spath/save=e12-a12,bend right=10] (e12.center) to (a12.center);

              \path[spath/save=merid-a0] (n-a0) to (s-a0);
              \path[spath/save=merid-yo] (n-yo) to (s-yo);
              \path[spath/save=merid-sum-p] (n-sum-p) to (s-sum-p);
              \path[spath/save=merid-E] (n-E) to (s-E);

              \tikzset{
                spath/split at intersections with={s-cup}{e0-e2},
                spath/split at intersections with={s-cup}{e01-a01},
                spath/split at intersections with={s-cup}{e12-a12},
                spath/split at intersections with={s-horiz}{e01-a01},
                spath/split at intersections with={s-horiz}{e12-a12},
                spath/split at intersections with={merid-yo}{n-cup},
                spath/split at intersections with={merid-sum-p}{n-cup},
                spath/split at intersections with={merid-sum-p}{n-horiz},
                spath/split at intersections with={merid-sum-p}{e01-a01},
                spath/insert gaps after components={s-cup}{1.5pt},
                spath/insert gaps after components={s-horiz}{1.5pt},
                spath/insert gaps after components={merid-sum-p}{3pt},
                spath/insert gaps after components={merid-yo}{3pt}
              }

              \draw[
                spath/use=n-horiz,
                spath/use=s-horiz,
                spath/use=n-cup,
                spath/use=s-cup,
                spath/use=e0-e2,
                spath/use=e01-a01,
                spath/use=e12-a12
              ];
              \draw[
                very densely dotted,
                spath/use=merid-a0,
                spath/use=merid-yo,
                spath/use=merid-sum-p,
                spath/use=merid-E,
              ];

              \path
                node[anchor=west] at (n-E.east) {\scriptsize$E$};

              \path
                node[anchor=north] at (a01.south) {\scriptsize$a_{01}$}
                node[anchor=north] at (a12.south) {\scriptsize$a_{12}$}
                node[anchor=north west] at (e2.south east) {\scriptsize$e_2$}
                node[anchor=south west] at (e0.north east) {\scriptsize$e_0$};
            \end{tikzpicture}
          \]
        \end{itemize}
        \noindent
        From the diagrammatic depictions above, the displayed associativity and unit laws are seen to follow immediately from the associativity and unit laws for composition of 3-cells in $\Gray$.
      \end{con}

      \begin{con}[Restriction]\label{con:opl-diag-restriction}
         For any fibred category $E\in\FIB{B}$ and fibration $p\colon A\to B$, let $e_1$ be a displayed object of $\OplDiag{p}\prn{E}$ over $a_1\in A$, and let $a_{01}\colon a_0\to a_1$ be an arrow in $A$. We define the \emph{restriction} of $e_1$ along $a_{01}$ to be the following displayed object $\Lift{\OplDiag{p}\prn{E}}{a_{01}}{e_1}$ over $a_0$:
         \[
          \Lift{\OplDiag{p}\prn{E}}{a_{01}}{e_1}
          \quad\coloneq\quad
          \begin{tikzpicture}[line width=.5pt,baseline=(current bounding box.center)]
            \begin{scope}[canvas is zx plane at y=0, yscale=-1, xscale=-1,transform shape]
              \CreateRect{3}{2.5}
              \path
                coordinate (a0) at (spath cs:north 0.25)
                coordinate (yo) at (spath cs:north 0.5)
                coordinate (sum-p) at (spath cs:north 0.75)
                coordinate (E) at (spath cs:south 0.5)
              ;
              \path[spath/save global=horiz] (yo) to (E);
              \path[spath/save global=cup] (sum-p) to[in=-90,out=-90,looseness=2.55] coordinate[dot] (e) (a0);
              \coordinate[dot] (a01) at (spath cs:cup 0.8);
            \end{scope}
            \draw[spath/use=horiz,spath/use=cup,];
            \path
              node[anchor=east] at (yo.west) {\scriptsize$\Yo{A}$}
              node[anchor=east] at (a0.west) {\scriptsize$a_0$}
              node[anchor=east] at (sum-p.west) {\scriptsize$\Sum{p}$}
              node[anchor=west] at (E.east) {\scriptsize$E$}
              node[anchor=north west] at (e.south east) {\scriptsize$e$}
              node[anchor=north] at (a01.south) {\scriptsize$a_{01}$};
          \end{tikzpicture}
         \]
      \end{con}

      \begin{obs}\label{obs:opl-diag-straightening}
        For any fibred category $E\in\FIB{B}$ and fibration $p\colon A\to B$, a displayed arrow $e_0\to e_2$ of $\OplDiag{p}\prn{E}$ over $a_{12}\circ a_{01}\colon a_0\to a_1\to a_2$
        \begin{equation*}
          \begin{tikzpicture}[diagram]
            \node (nw) {$e_0$};
            \node[below = 1.5cm of nw] (sw) {$a_0$};
            \node[right = of sw] (sc) {$a_1$};
            \node[right = of sc] (se) {$a_2$};
            \node[above = 1.5cm of se] (ne) {$e_2$};
            \draw[->] (sw) to node[below] {$a_{01}$} (sc);
            \draw[->] (sc) to node[below] {$a_{12}$} (se);
            \draw[|->] (nw) to (sw);
            \draw[|->] (ne) to (se);
            \draw[->] (nw) to (ne);
          \end{tikzpicture}
        \end{equation*}
        is by definition \emph{exactly} the same as a displayed arrow $e_0\to a_{12}^*e_2$ over $a_{01}\colon a_0\to a_1$
        \begin{equation*}
          \DiagramSquare{
            nw = e_0,
            ne = a_{12}^*e_2,
            sw = a_0,
            se = a_1,
            west/style = {|->},
            east/style = {|->},
            height = 1.5cm,
            south = a_{01},
          }
        \end{equation*}
        where $a_{12}^*e_2$ is the restriction of $e_2$ along $a_{12}\colon a_1\to a_2$ as in \autoref{con:opl-diag-restriction}.
      \end{obs}

      \begin{lem}[A split cleaving]\label{lem:opl-diag-splitting}
        Let $E\in\FIB{B}$ be a fibred category over $B$, and let $p\colon A\rightarrowtriangle B$ be a fibration. We may equip $\OplDiag{p}\prn{E}$ with a \emph{split cleaving} with cartesian lifts chosen as follows:
         \begin{equation*}
          \DiagramSquare{
            nw/style = pullback,
            nw = \Lift{\OplDiag{p}\prn{E}}{a_{01}}{e_1},
            ne = e_1,
            north = \LiftArr{\OplDiag{p}\prn{E}}{a_{01}}{e_1},
            sw = a_0,
            se = a_1,
            south = a_{01},
            west/style = {|->},
            east/style = {|->},
            height = 1.5cm,
          }
          \quad\coloneq\quad
          \begin{tikzpicture}[line width=.5pt,baseline=(current bounding box.center)]
            \begin{scope}[canvas is zx plane at y=1,yscale=-1,xscale=-1,transform shape]
              \CreateRect{3}{2.5}
              \path
                coordinate (n-a0) at (spath cs:north 0.25)
                coordinate (n-yo) at (spath cs:north 0.5)
                coordinate (n-sum-p) at (spath cs:north 0.75)
                coordinate (n-E) at (spath cs:south 0.5)
              ;
              \path[spath/save global=n-horiz] (n-yo) to (n-E);
              \path[spath/save global=n-cup] (n-sum-p) to[in=-90,out=-90,looseness=3] coordinate[dot] (n-e) (n-a0);
              \coordinate[dot] (n-a01) at (spath cs:n-cup 0.8);
            \end{scope}

            \begin{scope}[canvas is zx plane at y=0, yscale=-1, xscale=-1,transform shape]
              \CreateRect{3}{2.5}
              \path
                coordinate (s-a0) at (spath cs:north 0.25)
                coordinate (s-yo) at (spath cs:north 0.5)
                coordinate (s-sum-p) at (spath cs:north 0.75)
                coordinate (s-E) at (spath cs:south 0.5)
              ;
              \path[spath/save global=s-horiz] (s-yo) to (s-E);
              \path[spath/save global=s-cup] (s-sum-p) to[in=-90,out=-90,looseness=3] coordinate[dot] (s-e) (s-a0);
              \coordinate[dot] (s-a01) at (spath cs:s-cup 0.8);
            \end{scope}

            \path[spath/save=idn-e] (n-e) to (s-e);
            \path[spath/save=idn-a01] (n-a01) to (s-a01);

            \path[spath/save=merid-a0] (n-a0) to (s-a0);
            \path[spath/save=merid-yo] (n-yo) to (s-yo);
            \path[spath/save=merid-sum-p] (n-sum-p) to (s-sum-p);
            \path[spath/save=merid-E] (n-E) to (s-E);

            \tikzset{
              spath/split at intersections with={s-cup}{idn-e},
              spath/split at intersections with={s-cup}{idn-a01},
              spath/split at intersections with={s-horiz}{idn-a01},
              spath/split at intersections with={merid-yo}{n-cup},
              spath/split at intersections with={merid-sum-p}{n-cup},
              spath/split at intersections with={merid-sum-p}{n-horiz},
              spath/insert gaps after components={s-cup}{1.5pt},
              spath/insert gaps after components={s-horiz}{1.5pt},
              spath/insert gaps after components={merid-sum-p}{3pt},
              spath/insert gaps after components={merid-yo}{3pt}
            }

            \draw[
              spath/use=idn-e,
              spath/use=idn-a01,
              spath/use=n-horiz,
              spath/use=s-horiz,
              spath/use=n-cup,
              spath/use=s-cup,
            ];
            \draw[
              very densely dotted,
              spath/use=merid-a0,
              spath/use=merid-yo,
              spath/use=merid-sum-p,
              spath/use=merid-E,
            ];

            \path
              node[anchor=east] at (n-yo.west) {\scriptsize$\Yo{A}$}
              node[anchor=east] at (n-a0.west) {\scriptsize$a_0$}
              node[anchor=east] at (n-sum-p.west) {\scriptsize$\Sum{p}$}
              node[anchor=west] at (n-E.east) {\scriptsize$E$};

            \path
              node[anchor=north west] at (n-a01.south east) {\scriptsize$a_{01}$}
              node[anchor=south west] at (n-e.north east) {\scriptsize$e_1$};
          \end{tikzpicture}
        \end{equation*}
      \end{lem}

      \noindent In other words, the cartesian lift is given by an identity 2-cell in $\FIB{B}\prn{\OplSum{p}\prn{\Yo{A}\prn{a_0}},E}$ of the restriction $\Lift{\OplDiag{p}{E}}{a_{01}}{e_1}$.

      \begin{proof}
        To check that $\LiftArr{\OplDiag{p}\prn{E}}{a_{01}}{e}\colon \Lift{\OplDiag{p}\prn{E}}{a_{01}}{e} \to e$ as-defined is cartesian, we fix an arrow $a_{-10}\colon a_{-1}\to a_0$ and a displayed arrow $e_{-11}\colon e_{-1}\to e$ over $a_{01}\circ a_{-10}$, and we need to find a unique factorisation as depicted below:
        \begin{equation*}
          \begin{tikzpicture}[diagram]
            \SpliceDiagramSquare<l/>{
              nw = e_{-1},
              ne = \Lift{\OplDiag{p}\prn{E}}{a_{01}}{e},
              north = \exists!h,
              sw = a_{-1},
              se = a_0,
              south = a_{-10},
              north/style = {->,exists},
              north/node/style = upright desc,
              ne/style = pullback,
              height = 1.5cm,
              width = 2.5cm,
              west/style = {|->},
              east/style = {|->},
            }
            \SpliceDiagramSquare<r/>{
              glue = west,
              glue target = l/,
              width = 2.5cm,
              height = 1.5cm,
              north = \LiftArr{\OplDiag{p}\prn{E}}{a_{01}}{e},
              north/node/style = upright desc,
              ne = e,
              se = a_1,
              south = a_{01},
              east/style = {|->},
            }
            \draw[->,bend left=35] (l/nw) to node[above] {$e_{-11}$} (r/ne);
          \end{tikzpicture}
        \end{equation*}

        In fact, the displayed arrow $h\colon e_{-1} \to a_{01}^*e$ over $a_{-10}$ can be defined to have the same underlying 3-cell as $e_{-11}$ following \autoref{obs:opl-diag-straightening}, so we are done.
      \end{proof}

      \begin{cor}
        Let $E\in\FIB{B}$ be a fibred category over $B$, and let $p\colon A\rightarrowtriangle B$ be a fibration. In $\OplDiag{p}\prn{E}$, a displayed arrow $e_{01}\colon e_0\to e_1$ over $a_{01}\colon a_0\to a_1$ is \emph{cartesian} if and if the corresponding cell $e_{01}\colon e_0\to e_1\circ \OplSum{p}\prn{\Yo{A}\prn{a_{01}}}$ is invertible in $\FIB{B}\prn{\OplSum{p}\prn{\Yo{A}\prn{a_0}},E}$.
      \end{cor}

      The 2-functor $\OplDiag{p}\prn{E}_\bullet\colon A\Op\to\CAT$ corresponding to the split cleaving specified in \autoref{lem:opl-diag-splitting} is precisely the one foreshadowed in \S~\ref{sec:intro:oplax-sum-and-base-change}:
      \[
        \OplDiag{p}\prn{E}_\bullet = \FIB{B}\prn{\OplSum{p}\prn{\Yo{A}\prn{-}},E}
      \]
    \end{xsect}

    \begin{xsect}{Action of oplax base change on 1-~and~2-cells}
      Let $E$ and $F$ be two fibred categories over $B$. In this section, we construct the local hom functors $\prn{\OplDiag{p}}_{E,F}\colon \FIB{B}\prn{E,F}\to \FIB{A}\prn{\OplDiag{p}\prn{E},\OplDiag{p}\prn{F}}$.

      \begin{con}[Action of local hom functor on 0-cells]\label{con:oplax-base-change:local-hom-functor}
        Let $f\colon E\to F$ be a fibred functor over $B$; we define a displayed functor $\prn{\OplDiag{p}}_{E,F}\prn{f}\colon \OplDiag{p}\prn{E}\to \OplDiag{p}\prn{F}$ as follows:
        \begin{itemize}[label=$\triangleright$]
          \item A displayed object $e\in\OplDiag{p}\prn{E}$ over $a\in A$ is sent to the following composite fibred functor:
          \[
            \prn{\OplDiag{p}}_{E,F}\prn{f}\prn{e}
            \quad\coloneq\quad
            \begin{tikzpicture}[line width=.5pt,baseline=(current bounding box.center)]
              \begin{scope}[canvas is zx plane at y=0, yscale=-1, xscale=-1,transform shape]
                \CreateRect{3}{2.5}
                \path
                  coordinate (a) at (spath cs:north 0.25)
                  coordinate (yo) at (spath cs:north 0.5)
                  coordinate (sum-p) at (spath cs:north 0.75)
                  coordinate (F) at (spath cs:south 0.5)
                ;
                \path[spath/save global=horiz] (yo) to (F);
                \path[spath/save global=cup] (sum-p) to[in=-90,out=-90,looseness=2.55] coordinate[dot] (e) (a);
                \coordinate[dot] (f) at (spath cs:horiz 0.8);
              \end{scope}
              \draw[spath/use=horiz,spath/use=cup,];
              \path
                node[anchor=east] at (yo.west) {\scriptsize$\Yo{A}$}
                node[anchor=east] at (a.west) {\scriptsize$a$}
                node[anchor=east] at (sum-p.west) {\scriptsize$\Sum{p}$}
                node[anchor=west] at (F.east) {\scriptsize$F$}
                node[anchor=north west] at (e.south east) {\scriptsize$e$}
                node[anchor=north] at (f.south) {\scriptsize$f$};
            \end{tikzpicture}
          \]
          \item A displayed arrow $e_{01}\colon e_0\to e_1$ in $\OplDiag{p}\prn{E}$ over $a_{01}\colon a_0\to a_1$ is sent to its whiskering with $f$ as shown below:
          \[
            \prn{\OplDiag{p}}_{E,F}\prn{f}\prn{e_{01}}
            \quad\coloneq\quad
            \begin{tikzpicture}[line width=.5pt, baseline=(current bounding box.center)]
              \begin{scope}[canvas is zx plane at y=2,yscale=-1,xscale=-1,transform shape]
                \CreateRect{3}{2.5}
                \path
                  coordinate (n-a0) at (spath cs:north 0.25)
                  coordinate (n-yo) at (spath cs:north 0.5)
                  coordinate (n-sum-p) at (spath cs:north 0.75)
                  coordinate (n-F) at (spath cs:south 0.5)
                ;
                \path[spath/save global=n-horiz] (n-yo) to (n-F);
                \path[spath/save global=n-cup] (n-sum-p) to[in=-90,out=-90,looseness=2] coordinate[dot] (e0) (n-a0);
                \coordinate[dot] (n-f) at (spath cs:n-horiz 0.8);
              \end{scope}

              \begin{scope}[canvas is zx plane at y=0, yscale=-1, xscale=-1,transform shape]
                \CreateRect{3}{2.5}
                \path
                  coordinate (s-a0) at (spath cs:north 0.25)
                  coordinate (s-yo) at (spath cs:north 0.5)
                  coordinate (s-sum-p) at (spath cs:north 0.75)
                  coordinate (s-F) at (spath cs:south 0.5)
                ;
                \path[spath/save global=s-horiz] (s-yo) to (s-F);
                \path[spath/save global=s-cup] (s-sum-p) to[in=-90,out=-90,looseness=2] coordinate[dot] (e1) (s-a0);
                \coordinate[dot] (a01) at (spath cs:s-cup 0.8);
                \coordinate[dot] (s-f) at (spath cs:s-horiz 0.8);
              \end{scope}

              \path[spath/save=e0-e1] (e0.center) to coordinate[dot,label=right:$e_{01}$] (e01) (e1.center);
              \path[spath/save=e01-a01,bend right=10] (e01.center) to (a01.center);

              \path[spath/save=merid-a0] (n-a0) to (s-a0);
              \path[spath/save=merid-yo] (n-yo) to (s-yo);
              \path[spath/save=merid-sum-p] (n-sum-p) to (s-sum-p);
              \path[spath/save=merid-F] (n-F) to (s-F);
              \path[spath/save=merid-f] (n-f) to (s-f);

              \tikzset{
                spath/split at intersections with={s-cup}{e0-e1},
                spath/split at intersections with={s-cup}{e01-a01},
                spath/split at intersections with={s-horiz}{e01-a01},
                spath/split at intersections with={merid-yo}{n-cup},
                spath/split at intersections with={merid-sum-p}{n-cup},
                spath/split at intersections with={merid-sum-p}{n-horiz},
                spath/insert gaps after components={s-cup}{1.5pt},
                spath/insert gaps after components={s-horiz}{1.5pt},
                spath/insert gaps after components={merid-sum-p}{3pt},
                spath/insert gaps after components={merid-yo}{3pt}
              }

              \draw[
                spath/use=n-horiz,
                spath/use=s-horiz,
                spath/use=n-cup,
                spath/use=s-cup,
                spath/use=e0-e1,
                spath/use=e01-a01,
                spath/use=merid-f,
              ];
              \draw[
                very densely dotted,
                spath/use=merid-a0,
                spath/use=merid-yo,
                spath/use=merid-sum-p,
                spath/use=merid-F,
              ];

              \path
                node[anchor=east] at (n-yo.west) {\scriptsize$\Yo{A}$}
                node[anchor=east] at (n-a0.west) {\scriptsize$a_0$}
                node[anchor=east] at (n-sum-p.west) {\scriptsize$\Sum{p}$}
                node[anchor=west] at (n-F.east) {\scriptsize$F$}
                node[anchor=south] at (n-f.north) {\scriptsize$f$};

              \path
                node[anchor=north] at (a01.south) {\scriptsize$a_{01}$}
                node[anchor=north west] at (e1.south east) {\scriptsize$e_1$}
                node[anchor=south west] at (e0.north east) {\scriptsize$e_0$};
            \end{tikzpicture}
          \]
        \end{itemize}
      \end{con}

      \begin{obs}
        The displayed functor $\prn{\OplDiag{p}}_{E,F}\prn{f}\colon \OplDiag{p}\prn{E}\to \OplDiag{p}\prn{F}$ is split fibred in the sense that it preserves the cartesian lifts specified in \autoref{lem:opl-diag-splitting} on the nose.
      \end{obs}

      \begin{con}[Action of local hom functor on 1-cells]
        Let $f_{01}\colon f_0\to f_1$ be a vertical natural transformation of fibred functors $f_0,f_1\colon E\to F$ over $B$. The component of the functorial action
        \[\prn{\OplDiag{p}}_{E,F}^{f_0,f_1}\prn{f_{01}}\colon \prn{\OplDiag{p}}_{E,F}\prn{f_0}\to \prn{\OplDiag{p}}_{E,F}\prn{f_1}\] at $e\in \OplDiag{p}\prn{E}$ over $a\in A$ is the following whiskering:
        \[
          \prn{\OplDiag{p}}_{E,F}^{f_0,f_1}\prn{f_{01}}_e
          \quad\coloneq\quad
          \begin{tikzpicture}[line width=.5pt, baseline=(current bounding box.center)]
            \begin{scope}[canvas is zx plane at y=1,yscale=-1,xscale=-1,transform shape]
              \CreateRect{3}{2.5}
              \path
                coordinate (n-a) at (spath cs:north 0.25)
                coordinate (n-yo) at (spath cs:north 0.5)
                coordinate (n-sum-p) at (spath cs:north 0.75)
                coordinate (n-F) at (spath cs:south 0.5)
              ;
              \path[spath/save global=n-horiz] (n-yo) to (n-F);
              \path[spath/save global=n-cup] (n-sum-p) to[in=-90,out=-90,looseness=2.5] coordinate[dot] (n-e) (n-a);
              \coordinate[dot] (n-f) at (spath cs:n-horiz 0.7);
            \end{scope}
            \begin{scope}[canvas is zx plane at y=0,yscale=-1,xscale=-1,transform shape]
              \CreateRect{3}{2.5}
              \path
                coordinate (s-a) at (spath cs:north 0.25)
                coordinate (s-yo) at (spath cs:north 0.5)
                coordinate (s-sum-p) at (spath cs:north 0.75)
                coordinate (s-F) at (spath cs:south 0.5)
              ;
              \path[spath/save global=s-horiz] (s-yo) to (s-F);
              \path[spath/save global=s-cup] (s-sum-p) to[in=-90,out=-90,looseness=2.5] coordinate[dot] (s-e) (s-a);
              \coordinate[dot] (s-f) at (spath cs:s-horiz 0.7);
            \end{scope}

            \path[spath/save=merid-e] (n-e) to (s-e);
            \path[spath/save=merid-a] (n-a) to (s-a);
            \path[spath/save=merid-yo] (n-yo) to (s-yo);
            \path[spath/save=merid-sum-p] (n-sum-p) to (s-sum-p);
            \path[spath/save=merid-F] (n-F) to (s-F);
            \path[spath/save=merid-f] (n-f) to coordinate[dot] (f01) (s-f);

            \tikzset{
              spath/split at intersections with={s-cup}{merid-e},
              spath/split at intersections with={merid-yo}{n-cup},
              spath/split at intersections with={merid-sum-p}{n-cup},
              spath/split at intersections with={merid-sum-p}{n-horiz},
              spath/insert gaps after components={s-cup}{1.5pt},
              spath/insert gaps after components={s-horiz}{1.5pt},
              spath/insert gaps after components={merid-sum-p}{3pt},
              spath/insert gaps after components={merid-yo}{3pt}
            }

            \draw[
              very densely dotted,
              spath/use=merid-a,
              spath/use=merid-yo,
              spath/use=merid-sum-p,
              spath/use=merid-F,
            ];

            \draw[
              spath/use=n-horiz,
              spath/use=s-horiz,
              spath/use=n-cup,
              spath/use=s-cup,
              spath/use=merid-e,
              spath/use=merid-f
            ];
            \path
              node[anchor=east] at (n-yo.west) {\scriptsize$\Yo{A}$}
              node[anchor=east] at (n-a.west) {\scriptsize$a$}
              node[anchor=east] at (n-sum-p.west) {\scriptsize$\Sum{p}$}
              node[anchor=west] at (n-F.east) {\scriptsize$F$}
              node[anchor=south west] at (n-e.north east) {\scriptsize$e$}
              node[anchor=south] at (n-f.north) {\scriptsize$f_0$}
              node[anchor=north] at (s-f.south) {\scriptsize$f_1$}
              node[anchor=west] at (f01.east) {$f_{01}$}
              ;
          \end{tikzpicture}
        \]

        We further verify the following naturality square for each $e_{01}\colon e_0\to e_1$ in $\OplDiag{p}\prn{E}$ over $a_{01}\colon a_0\to a_1$:
        \[
          \DiagramSquare{
            nw = \prn{\OplDiag{p}}_{E,F}\prn{f_0}\prn{e_0},
            ne = \prn{\OplDiag{p}}_{E,F}\prn{f_1}\prn{e_0},
            sw = \prn{\OplDiag{p}}_{E,F}\prn{f_0}\prn{e_1},
            se = \prn{\OplDiag{p}}_{E,F}\prn{f_1}\prn{e_1},
            north = \prn{\OplDiag{p}}_{E,F}^{f_0,f_1}\prn{f_{01}}_{e_0},
            south = \prn{\OplDiag{p}}_{E,F}^{f_0,f_1}\prn{f_{01}}_{e_1},
            west = \prn{\OplDiag{p}}_{E,F}\prn{f_0}\prn{e_{01}},
            east = \prn{\OplDiag{p}}_{E,F}\prn{f_1}\prn{e_{01}},
            width = 7cm,
          }
        \]

        That this square commutes follows from the interchange of the 3-cells $f_{01}$ and $e_{01}$ in the surface diagrams representing the two composites:
        \[
          \begin{tikzpicture}[line width=.5pt, baseline=(current bounding box.center)]
            \begin{scope}[canvas is zx plane at y=2,yscale=-1,xscale=-1,transform shape]
              \CreateRect{3}{2.5}
              \path
                coordinate (n-a0) at (spath cs:north 0.25)
                coordinate (n-yo) at (spath cs:north 0.5)
                coordinate (n-sum-p) at (spath cs:north 0.75)
                coordinate (n-F) at (spath cs:south 0.5)
              ;
              \path[spath/save global=n-horiz] (n-yo) to (n-F);
              \path[spath/save global=n-cup] (n-sum-p) to[in=-90,out=-90,looseness=2] coordinate[dot] (e0) (n-a0);
              \coordinate[dot] (n-f) at (spath cs:n-horiz 0.8);
            \end{scope}

            \begin{scope}[canvas is zx plane at y=0, yscale=-1, xscale=-1,transform shape]
              \CreateRect{3}{2.5}
              \path
                coordinate (s-a0) at (spath cs:north 0.25)
                coordinate (s-yo) at (spath cs:north 0.5)
                coordinate (s-sum-p) at (spath cs:north 0.75)
                coordinate (s-F) at (spath cs:south 0.5)
              ;
              \path[spath/save global=s-horiz] (s-yo) to (s-F);
              \path[spath/save global=s-cup] (s-sum-p) to[in=-90,out=-90,looseness=2] coordinate[dot] (e1) (s-a0);
              \coordinate[dot] (a01) at (spath cs:s-cup 0.8);
              \coordinate[dot] (s-f) at (spath cs:s-horiz 0.8);
            \end{scope}

            \path[spath/save=e0-e1] (e0.center) to coordinate[dot,near end,label=right:$e_{01}$] (e01) (e1.center);
            \path (n-f.center) to coordinate[dot,near start,label=left:$f_{01}$] (f01) (s-f.center);
            \path[spath/save=e01-a01,bend right=10] (e01.center) to (a01.center);

            \path[spath/save=merid-a0] (n-a0) to (s-a0);
            \path[spath/save=merid-yo] (n-yo) to (s-yo);
            \path[spath/save=merid-sum-p] (n-sum-p) to (s-sum-p);
            \path[spath/save=merid-F] (n-F) to (s-F);
            \path[spath/save=merid-f] (n-f) to (s-f);

            \tikzset{
              spath/split at intersections with={s-cup}{e0-e1},
              spath/split at intersections with={s-cup}{e01-a01},
              spath/split at intersections with={s-horiz}{e01-a01},
              spath/split at intersections with={merid-yo}{n-cup},
              spath/split at intersections with={merid-sum-p}{n-cup},
              spath/split at intersections with={merid-sum-p}{n-horiz},
              spath/insert gaps after components={s-cup}{1.5pt},
              spath/insert gaps after components={s-horiz}{1.5pt},
              spath/insert gaps after components={merid-sum-p}{3pt},
              spath/insert gaps after components={merid-yo}{3pt}
            }

            \draw[
              spath/use=n-horiz,
              spath/use=s-horiz,
              spath/use=n-cup,
              spath/use=s-cup,
              spath/use=e0-e1,
              spath/use=e01-a01,
              spath/use=merid-f,
            ];
            \draw[
              very densely dotted,
              spath/use=merid-a0,
              spath/use=merid-yo,
              spath/use=merid-sum-p,
              spath/use=merid-F,
            ];

            \path
              node[anchor=south] at (n-f.north) {\scriptsize$f_0$}
              node[anchor=north] at (s-f.south) {\scriptsize$f_1$}
            ;

            \path
              node[anchor=north] at (a01.south) {\scriptsize$a_{01}$}
              node[anchor=north west] at (e1.south east) {\scriptsize$e_1$}
              node[anchor=south west] at (e0.north east) {\scriptsize$e_0$};
          \end{tikzpicture}
          \quad=\quad
          \begin{tikzpicture}[line width=.5pt, baseline=(current bounding box.center)]
            \begin{scope}[canvas is zx plane at y=2,yscale=-1,xscale=-1,transform shape]
              \CreateRect{3}{2.5}
              \path
                coordinate (n-a0) at (spath cs:north 0.25)
                coordinate (n-yo) at (spath cs:north 0.5)
                coordinate (n-sum-p) at (spath cs:north 0.75)
                coordinate (n-F) at (spath cs:south 0.5)
              ;
              \path[spath/save global=n-horiz] (n-yo) to (n-F);
              \path[spath/save global=n-cup] (n-sum-p) to[in=-90,out=-90,looseness=2] coordinate[dot] (e0) (n-a0);
              \coordinate[dot] (n-f) at (spath cs:n-horiz 0.8);
            \end{scope}

            \begin{scope}[canvas is zx plane at y=0, yscale=-1, xscale=-1,transform shape]
              \CreateRect{3}{2.5}
              \path
                coordinate (s-a0) at (spath cs:north 0.25)
                coordinate (s-yo) at (spath cs:north 0.5)
                coordinate (s-sum-p) at (spath cs:north 0.75)
                coordinate (s-F) at (spath cs:south 0.5)
              ;
              \path[spath/save global=s-horiz] (s-yo) to (s-F);
              \path[spath/save global=s-cup] (s-sum-p) to[in=-90,out=-90,looseness=2] coordinate[dot] (e1) (s-a0);
              \coordinate[dot] (a01) at (spath cs:s-cup 0.8);
              \coordinate[dot] (s-f) at (spath cs:s-horiz 0.8);
            \end{scope}

            \path[spath/save=e0-e1] (e0.center) to coordinate[dot,near start,label=right:$e_{01}$] (e01) (e1.center);
            \path (n-f.center) to coordinate[dot,near end,label=left:$f_{01}$] (f01) (s-f.center);
            \path[spath/save=e01-a01,bend right=10] (e01.center) to (a01.center);

            \path[spath/save=merid-a0] (n-a0) to (s-a0);
            \path[spath/save=merid-yo] (n-yo) to (s-yo);
            \path[spath/save=merid-sum-p] (n-sum-p) to (s-sum-p);
            \path[spath/save=merid-F] (n-F) to (s-F);
            \path[spath/save=merid-f] (n-f) to (s-f);

            \tikzset{
              spath/split at intersections with={s-cup}{e0-e1},
              spath/split at intersections with={s-cup}{e01-a01},
              spath/split at intersections with={s-horiz}{e01-a01},
              spath/split at intersections with={merid-yo}{n-cup},
              spath/split at intersections with={merid-sum-p}{n-cup},
              spath/split at intersections with={merid-sum-p}{n-horiz},
              spath/insert gaps after components={s-cup}{1.5pt},
              spath/insert gaps after components={s-horiz}{1.5pt},
              spath/insert gaps after components={merid-sum-p}{3pt},
              spath/insert gaps after components={merid-yo}{3pt}
            }

            \draw[
              spath/use=n-horiz,
              spath/use=s-horiz,
              spath/use=n-cup,
              spath/use=s-cup,
              spath/use=e0-e1,
              spath/use=e01-a01,
              spath/use=merid-f,
            ];
            \draw[
              very densely dotted,
              spath/use=merid-a0,
              spath/use=merid-yo,
              spath/use=merid-sum-p,
              spath/use=merid-F,
            ];

            \path
              node[anchor=south] at (n-f.north) {\scriptsize$f_0$}
              node[anchor=north] at (s-f.south) {\scriptsize$f_1$}
            ;

            \path
              node[anchor=north] at (a01.south) {\scriptsize$a_{01}$}
              node[anchor=north west] at (e1.south east) {\scriptsize$e_1$}
              node[anchor=south west] at (e0.north east) {\scriptsize$e_0$};
          \end{tikzpicture}
        \]
      \end{con}

      We have now constructed each of the local hom functors $\prn{\OplDiag{p}}_{E,F}\colon \FIB{B}\prn{E,F}\to \FIB{A}\prn{\OplDiag{p}\prn{E},\OplDiag{p}\prn{F}}$.
    \end{xsect}

    \begin{xsect}{Respect of local hom functors for composition and identity}
      In order to conclude that the actions of oplax base change on 0-, 1-, and 2-cells that we have defined above assemble into a 2-functor, it remains to check that the assignment of local hom functors respects horizontal composition and identity data.
      In what follows, we shall write $\Gamma_{\mathfrak{C}}$ and $\mathrm{J}_{\mathfrak{C}}$ for the respective horizontal composition and identity data of a 2-category $\mathfrak{C}$.

      \begin{lem}\label{lem:local-hom-functors-preserve-composition}
        For any three fibred categories $E,F,G\in\FIB{B}$ over $B$, the square below commutes strictly.
        \[
          \DiagramSquare{
            nw = \FIB{B}\prn{E,F}\times\FIB{B}\prn{F,G},
            ne = \FIB{B}\prn{E,G},
            sw = {
             \begin{array}{l}
              \FIB{A}\prn{\OplDiag{p}\prn{E},\OplDiag{p}\prn{F}}\\
              {}\times
              \FIB{A}\prn{\OplDiag{p}\prn{F},\OplDiag{p}\prn{G}}
             \end{array}
            },
            se = \FIB{A}\prn{\OplDiag{p}\prn{E},\OplDiag{p}\prn{G}},
            north = \Gamma_{\FIB{B}}^{E,F,G},
            south = \Gamma_{\FIB{A}}^{\OplDiag{p}\prn{E},\OplDiag{p}\prn{F},\OplDiag{p}\prn{G}},
            west = \prn{\OplDiag{p}}_{E,F}\times\prn{\OplDiag{p}}_{F,G},
            east = \prn{\OplDiag{p}}_{E,G},
            width = 9cm,
          }
        \]
      \end{lem}

      \begin{lem}\label{lem:local-hom-functors-preserve-identities}
        Let $p\colon A\rightarrowtriangle B$ be a fibration. For any fibred category $E\in \FIB{B}$, the following triangle commutes strictly:
        \begin{equation*}
          \begin{tikzpicture}[diagram]
            \node (nw) {$\mathbf{1}_{\CAT}$};
            \node[below right = 4cm of nw] (s) {$\FIB{A}\prn{\OplDiag{p}\prn{E},\OplDiag{p}\prn{E}}$};
            \node[above right = 4cm of s] (ne) {$\FIB{B}\prn{E,E}$};
            \draw[->] (nw) to node[above] {$\mathrm{J}_{\FIB{B}}^{E}$} (ne);
            \draw[->] (nw) to node[below left] {$\mathrm{J}_{\FIB{A}}^{\OplDiag{p}\prn{E}}$} (s);
            \draw[->] (ne) to node[below right] {$\prn{\OplDiag{p}}_{E,E}$} (s);
          \end{tikzpicture}
        \end{equation*}
      \end{lem}
    \end{xsect}

    \begin{cor}
      Given any fibration $p\colon A\rightarrowtriangle B$, Constructions~\ref{con:oplax-base-change} and \ref{con:oplax-base-change:local-hom-functor} assemble into a 2-functor $\OplDiag{p}\colon \FIB{B}\to\FIB{A}$.
    \end{cor}
  \end{xsect}

  \begin{xsect}{Oplax base change as a right pseudo-adjoint}

    Let $p\colon A\rightarrowtriangle B$ be a fibration. We plan to exhibit the oplax base change 2-functor $\OplDiag{p}\colon\FIB{B}\to\FIB{A}$ as the right pseudo-adjoint to the oplax sum $\OplSum{p}\colon\FIB{A}\to\FIB{B}$. There are multiple ways to present pseudo-adjunctions, but we will prefer the one based on a pseudo-natural equivalence of hom-categories. In other words:

    \begin{itemize}[label=$\triangleright$]
      \item We must assign to fibred categories $E\in\FIB{A}$ and $F\in \FIB{B}$ equivalences \[{\sharp_{E,F}}\colon \FIB{A}\prn{E,\OplDiag{p}\prn{F}}\simeq \FIB{B}\prn{\OplSum{p}\prn{E},F}\text{.}\]
      \item To fibred functors $E_{10}\colon E_1\to E_0$ and $F_{01}\colon F_0\to F_1$ we must assign invertible pseudo-naturality cells as follows:
        \[
          \begin{tikzpicture}[line width=.5pt,baseline=(current bounding box.center)]
            \begin{scope}[canvas is zx plane at y=3,yscale=-1,xscale=-1,transform shape]
              \CreateRect{4}{4}
              \path
                coordinate (n-E-DpF) at (spath cs:north 0.5)
                coordinate (n-SpE-F) at (spath cs:south 0.5)
              ;
              \path[spath/save global=n-horiz] (n-E-DpF) to (n-SpE-F);
              \path
                coordinate[dot] (n-tr) at (spath cs:n-horiz 0.4)
                coordinate[dot] (n-sum) at (spath cs:n-horiz 0.85);
            \end{scope}

            \begin{scope}[canvas is zx plane at y=0, yscale=-1, xscale=-1,transform shape]
              \CreateRect{4}{4}
              \path
                coordinate (s-E-DpF) at (spath cs:north 0.5)
                coordinate (s-SpE-F) at (spath cs:south 0.5)
              ;
              \path[spath/save global=s-horiz] (s-E-DpF) to (s-SpE-F);
              \path
                coordinate[dot] (s-tr) at (spath cs:s-horiz 0.6)
                coordinate[dot] (s-diag) at (spath cs:s-horiz 0.15);
            \end{scope}

            \draw
              node[anchor=south] at (n-tr.north) {\scriptsize$\strut\sharp_{E_0,F_0}$}
              node[anchor=south] at (n-sum.north) {\scriptsize$\strut\OplSum{p}\prn{E_{10}}^*$}
              node[anchor=north] at (s-tr.south) {\scriptsize$\strut\sharp_{E_1,F_1}$}
              node[anchor=north] at (s-diag.south) {\scriptsize$\strut\OplDiag{p}\prn{F_{01}}_!$}
              ;

            \draw[spath/save=ns-tr] (n-tr) to (s-tr);
            \draw[spath/save=ns-aux] (n-sum) to[in=90,out=-90] (s-diag);

            \path[name intersections={of=ns-tr and ns-aux}]
              coordinate[dot,label={340:$\sharp_{E_{10},F_{01}}$}] (tr-nat) at (intersection-1)
            ;

            \draw[
              spath/use=n-horiz,
              spath/use=s-horiz,
            ];

            \draw[densely dotted] (n-E-DpF) to node[anchor=east] {\scriptsize$\FIB{A}\prn{E_0,\OplDiag{p}\prn{F_0}}$} (s-E-DpF);
            \draw[densely dotted] (n-SpE-F) to node[anchor=west] {\scriptsize$\FIB{A}\prn{\OplSum{p}\prn{E_1},F_1}$} (s-SpE-F);
          \end{tikzpicture}
        \]
      \item The pseudo-naturality datum $\sharp_{1_E,1_F}$ needs to be equal to the identity 2-cell on \[ \sharp_{E,F}\colon \FIB{A}\prn{E,\OplDiag{p}\prn{F}}\to \FIB{B}\prn{\OplSum{p}\prn{E},F}\text{.}\]
      \item For $E_{21}\colon E_2\to E_1$, $E_{10}\colon E_1\to E_0$, $F_{01}\colon E_0\to E_1$ and $F_{12}\colon E_1\to E_2$, the pseudonaturality datum $\sharp_{E_{10}\circ E_{21},F_{12}\circ F_{01}}$ must be equal to the following:
        \[
          \begin{tikzpicture}[line width=.5pt,baseline=(current bounding box.center)]
            \begin{scope}[canvas is zx plane at y=4,yscale=-1,xscale=-1,transform shape]
              \CreateRect{4}{6}
              \path
                coordinate (n-E-DpF) at (spath cs:north 0.5)
                coordinate (n-SpE-F) at (spath cs:south 0.5)
              ;
              \path[spath/save global=n-horiz] (n-E-DpF) to (n-SpE-F);
              \path
                coordinate[dot] (n-tr) at (spath cs:n-horiz 0.3)
                coordinate[dot] (n-sum-10) at (spath cs:n-horiz 0.6)
                coordinate[dot] (n-sum-21) at (spath cs:n-horiz 0.9);
                ;
            \end{scope}

            \begin{scope}[canvas is zx plane at y=0, yscale=-1, xscale=-1,transform shape]
              \CreateRect{4}{6}
              \path
                coordinate (s-E-DpF) at (spath cs:north 0.5)
                coordinate (s-SpE-F) at (spath cs:south 0.5)
              ;
              \path[spath/save global=s-horiz] (s-E-DpF) to (s-SpE-F);
              \path
                coordinate[dot] (s-tr) at (spath cs:s-horiz 0.7)
                coordinate[dot] (s-diag-01) at (spath cs:s-horiz 0.1)
                coordinate[dot] (s-diag-12) at (spath cs:s-horiz 0.4)
                ;
            \end{scope}

            \draw
              node[anchor=south] at (n-tr.north) {\scriptsize$\strut\sharp_{E_0,F_0}$}
              node[anchor=south] at (n-sum-10.north) {\scriptsize$\strut\OplSum{p}\prn{E_{10}}^*$}
              node[anchor=south] at (n-sum-21.north) {\scriptsize$\strut\OplSum{p}\prn{E_{21}}^*$}
              node[anchor=north] at (s-tr.south) {\scriptsize$\strut\sharp_{E_1,F_1}$}
              node[anchor=north] at (s-diag-01.south) {\scriptsize$\strut\OplDiag{p}\prn{F_{01}}_!$}
              node[anchor=north] at (s-diag-12.south) {\scriptsize$\strut\OplDiag{p}\prn{F_{12}}_!$}
              ;

            \draw[spath/save=ns-tr] (n-tr) to (s-tr);
            \draw[spath/save=ns-aux-01] (n-sum-10) to[in=90,out=-90] (s-diag-01);
            \draw[spath/save=ns-aux-12] (n-sum-21) to[in=90,out=-90] (s-diag-12);

            \path[name intersections={of=ns-tr and ns-aux-01}]
              coordinate[dot,label={170:$\sharp_{E_{10},F_{01}}$}] (tr-nat) at (intersection-1)
            ;
            \path[name intersections={of=ns-tr and ns-aux-12}]
              coordinate[dot,label={[label distance=.2cm]0:$\sharp_{E_{21},F_{12}}$}] (tr-nat) at (intersection-1)
            ;
            \draw[
              spath/use=n-horiz,
              spath/use=s-horiz,
            ];

            \draw[densely dotted] (n-E-DpF) to node[anchor=east] {\scriptsize$\FIB{A}\prn{E_0,\OplDiag{p}\prn{F_0}}$} (s-E-DpF);
            \draw[densely dotted] (n-SpE-F) to node[anchor=west] {\scriptsize$\FIB{A}\prn{\OplSum{p}\prn{E_2},F_2}$} (s-SpE-F);
          \end{tikzpicture}
        \]
    \end{itemize}

    \begin{xsect}{Transpose functors}
      Let $E\in\FIB{A}$ and $F\in\FIB{B}$ be fibred categories.
      In this section, we will define the transposition functor ${\sharp_{E,F}}\colon \FIB{A}\prn{E,\OplDiag{p}\prn{F}} \to \FIB{B}\prn{\OplSum{p}\prn{E},F}$ and show that it is an equivalence.

      \begin{con}[Action of transposition on objects]
        Let $\phi\colon E \to \OplDiag{p}\prn{F}$ be a fibred functor over $A$. We first define a displayed functor $\phi^\sharp\colon \OplSum{p}\prn{E}\to F$ over $B$ acting as follows:
        \begin{itemize}[label=$\triangleright$]
          \item \emph{On objects.} Fix a displayed object $\gl{a,e}\in\OplSum{p}\prn{E}$ over $b\in B$, so that $e$ lies over $a$ and $p\prn{a}=b$. We define $\phi^\sharp_b\gl{a,e}$ to be the component $\phi_a\prn{e}\gl{a,1_a}$.
          \item \emph{On arrows.} Fix a displayed arrow $\gl{a_{01},e_{01}}$ in $\OplSum{p}\prn{E}$ over $b_{01}\colon b_0\to b_1$ so that $p\prn{a_{01}}=b_{01}$ and $e_{01}$ lies over $a_{01}$. We define \[ \phi^\sharp_{b_{01}}\gl{a_{01},e_{01}}\colon \phi_{a_0}\prn{e_0}\gl{a_0,1_{a_0}}\to \phi_{a_1}\prn{e_1}\gl{a_1,1_{a_1}}\] as follows:
          \begin{itemize}
            \item Using the displayed functorial action of $\phi$, we have a displayed arrow $\phi_{a_{01}}\prn{e_{01}}\colon \phi_{a_0}\prn{e_0}\to \phi_{a_1}\prn{e_1}$ in $\OplDiag{p}\prn{F}$ over $a_{01}$.
            \item The underlying 2-cell of the above is thus an arrow \[\phi_{a_{01}}\prn{e_{01}}\colon \phi_{a_0}\prn{e_0}\to \phi_{a_1}\prn{e_1} \circ \OplSum{p}\prn{\Yo{A}\prn{a_{01}}}\] in $\FIB{B}\prn{\OplSum{p}\prn{\Yo{A}\prn{a_0}},F}$. Instantiating at the displayed object $\gl{a_0,1_{a_0}}$ in $\OplSum{p}\prn{\Yo{A}\prn{a_0}}$ over $b_0$, we obtain
            \[\phi_{a_{01}}\prn{e_{01}}\gl{a_0,1_{a_0}}\colon \phi_{a_0}\prn{e_0}\gl{a_0,1_{a_0}}\to \phi_{a_1}\prn{e_1}\gl{a_0,a_{01}}\text{.}\]
            \item Finally, we compose with the functorial action of $\phi_{a_1}\prn{e_1}$ on the terminal map \[!_{\gl{a_0,a_{01}}}\colon \gl{a_0,a_{01}}\to \gl{a_1,1_{a_1}}\] in $\OplSum{p}\prn{\Yo{A}\prn{a_1}}$ over $b_{01}\colon b_0\to b_1$.
          \end{itemize}
          \noindent
          All in all, we define $\phi^\sharp_{b_{01}}\gl{a_{01},e_{01}} = \phi_{a_1}\prn{e_1}\prn{!_{\gl{a_0,a_{01}}}}\circ \phi_{a_{01}}\prn{e_{01}}\gl{a_0,1_{a_0}}$.
        \end{itemize}
      \end{con}

      \begin{lem}
        The displayed functor $\phi^\sharp\colon \OplSum{p}\prn{E}\to F$ is fibred, \ie sends cartesian arrows to cartesian arrows.
      \end{lem}
      \begin{proof}
        We first fix the data of a cartesian arrow in $\OplSum{p}\prn{E}$.
        \[
          \begin{tikzpicture}[diagram]
            \SpliceDiagramSquare<top/>{
              height = 1.5cm,
              west/style = {|->},
              east/style = {|->},
              nw/style = pullback,
              sw/style = pullback,
              nw = e_0,
              ne = e_1,
              north = e_{01},
              sw = a_0,
              se = a_1,
              south = a_{01},
              south/node/style = upright desc,
            }
            \node[below = 1.5cm of top/sw] (bot/sw) {$b_0$};
            \node[below = 1.5cm of top/se] (bot/se) {$b_1$};
            \draw[->] (bot/sw) to node[below] {$b_{01}$} (bot/se);
            \draw[{|->}] (top/sw) to (bot/sw);
            \draw[{|->}] (top/se) to (bot/se);
          \end{tikzpicture}
        \]

        We wish to check that the composite $\phi^\sharp_{b_{01}}\gl{a_{01},e_{01}} = \phi_{a_1}\prn{e_1}\prn{!_{\gl{a_0,a_{01}}}}\circ \phi_{a_{01}}\prn{e_{01}}\gl{a_0,1_{a_0}}$ is cartesian over $b_{01}$.
        \[
          \begin{tikzpicture}[diagram]
            \SpliceDiagramSquare<sq/>{
              width = 7cm,
              nw = \phi_{a_1}\prn{e_1}\gl{a_{0},a_{01}},
              ne = \phi_{a_1}\prn{e_1}\gl{a_1,1_{a_1}},
              north = \phi_{a_1}\prn{e_1}\prn{!_{\gl{a_0,a_{01}}}},
              north/node/style = upright desc,
              sw = b_0,
              se = b_1,
              south = b_{01},
              west/style = {|->},
              east/style = {|->},
              height = 1.5cm,
            }
            \node[above = of sq/nw] (nw) {$\phi_{a_0}\prn{e_0}\gl{a_0,1_{a_0}}$};
            \draw[->] (nw) to node[left] {$\phi_{a_{01}}\prn{e_{01}}\gl{a_0,1_{a_0}}$} (sq/nw);
            \draw[->] (nw) to node[sloped,above] {$\phi^\sharp_{b_{01}}\gl{a_{01},e_{01}}$} (sq/ne);
          \end{tikzpicture}
        \]

        We first observe that the universal map $!_{\gl{a_0,a_{01}}}\colon \gl{a_0,a_{01}}\to \gl{a_1,1_{a_1}}$ is cartesian in $\OplSum{p}\prn{\Yo{A}\prn{a_1}}$ over $b_{01}\colon b_0\to b_1$ because $a_{01}$ is cartesian over $b_{01}$:
        \[
          \DiagramSquare{
            nw = \gl{a_0,a_{01}},
            ne = \gl{a_1,1_{a_1}},
            sw = b_0,
            se = b_1,
            south = b_{01},
            north = !_{\gl{a_0,a_{01}}},
            nw/style = pullback,
            west/style = {|->},
            east/style = {|->},
            height = 1.5cm,
            width = 3cm,
          }
        \]

        Because $\phi_{a_1}\prn{e_1}\colon \OplSum{p}\prn{\Yo{A}\prn{a_1}}\to F$ is a fibred functor, the induced functorial arrow below is likewise cartesian in $F$ over $b_{01}\colon b_0\to b_1$:
        \[
          \DiagramSquare{
            width = 6cm,
            nw = \phi_{a_1}\prn{e_1}\gl{a_{0},a_{01}},
            ne = \phi_{a_1}\prn{e_1}\gl{a_1,1_{a_1}},
            north = \phi_{a_1}\prn{e_1}\prn{!_{\gl{a_0,a_{01}}}},
            sw = b_0,
            se = b_1,
            south = b_{01},
            west/style = {|->},
            east/style = {|->},
            height = 1.5cm,
            nw/style = pullback,
          }
        \]

        Because $\phi\colon E\to \OplDiag{p}\prn{F}$ is itself a fibred functor over $A$, the following must be cartesian in $\OplDiag{p}\prn{F}$ over $a_{01}$:
        \[
          \DiagramSquare{
            west/style = {|->},
            east/style = {|->},
            height = 1.5cm,
            width = 4cm,
            nw/style = pullback,
            nw = \phi_{a_0}\prn{e_0},
            ne = \phi_{a_1}\prn{e_1},
            north = \phi_{a_{01}}\prn{e_{01}},
            sw = a_0,
            se = a_1,
            south = a_{01},
          }
        \]

        Recall that we have split $\OplDiag{p}\prn{F}$ explicitly in \autoref{lem:opl-diag-splitting} by setting $\Lift{\OplDiag{p}\prn{F}}{a_{01}}{\phi_{a_1}\prn{e_1}} = \phi_{a_1}\prn{e_1}\circ\OplSum{p}\prn{\Yo{A}\prn{a_{01}}}$. Therefore, we have the following vertical cartesian isomorphism in $\OplDiag{p}\prn{F}$ where the composite cartesian arrow is, at the level of fibred natural transformations, the identity cell:
        \[
          \begin{tikzpicture}[diagram]
            \SpliceDiagramSquare<sq/>{
              nw = \phi_{a_0}\prn{e_0},
              ne = \phi_{a_1}\prn{e_1},
              north = \phi_{a_{01}}\prn{e_{01}},
              nw/style = pullback,
              sw = a_0,
              se = a_1,
              south = a_{01},
              north/node/style = upright desc,
              width = 5cm,
              height = 1.5cm,
              west/style = {|->},
              east/style = {|->},
            }
            \node[above = of sq/nw] (nw) {$\phi_{a_1}\prn{e_1}\circ\OplSum{p}\prn{\Yo{A}\prn{a_{01}}}$};
            \draw[->,exists] (nw) to node[left] {$\cong$} (sq/nw);
            \draw[->] (nw) to node[sloped,above] {$\LiftArr{\OplDiag{p}\prn{F}}{a_{01}}{\phi_{a_1}\prn{e_1}}$} (sq/ne);
          \end{tikzpicture}
        \]

        Instantiating the triangle above at $\gl{a_0,1_{a_0}}$ we therefore have the following in $F$:
        \[
          \begin{tikzpicture}[diagram]
            \node (nw) {$\phi_{a_1}\prn{e_1}\gl{a_0,a_{01}}$};
            \node[below = of nw] (sw) {$\phi_{a_0,a_{1_{a_0}}}$};
            \node[right = 5cm of sw] (se) {$\phi_{a_1}\prn{e_1}\gl{a_0,a_{01}}$};
            \draw[->] (nw) to node[left] {$\cong$} (sw);
            \draw[double] (nw) to (se);
            \draw[->] (sw) to node[below] {$\phi_{a_{01}}\prn{e_{01}}\gl{a_0,1_{a_0}}$} (se);
          \end{tikzpicture}
        \]

        To see that $\phi^\sharp_{b_{01}}\gl{a_{01},e_{01}}$ is cartesian, it is enough to see that its precomposition by the vertical cartesian isomorphism $\phi_{a_1}\prn{e_1}\gl{a_0,a_{01}}\cong \phi_{a_0}\prn{e_0}\gl{a_0,1_{a_0}}$ is cartesian. But this composite is equal to the cartesian arrow $\phi_{a_1}\prn{e_1}\prn{!_{\gl{a_0,a_{01}}}}$ so we are done.
      \end{proof}

      \begin{con}[Action of transposition on arrows]
        Let $\phi_{01}\colon \phi_0\to\phi_1$ be a natural transformation of fibred functors $\phi_0,\phi_1\colon E\to \OplDiag{p}\prn{F}$ over $A$. We shall define a natural transformation $\phi_{01}^\sharp\colon \phi_0^\sharp \to \phi_1^\sharp$ in $\FIB{B}\prn{\OplSum{p}\prn{E},F}$.
        Given a displayed object $\gl{a,e}\in\OplSum{p}\prn{E}$ over $b\in B$, so that $e$ lies over $a$ and $p\prn{a}=b$, we define the component $\prn{\phi^\sharp_{01}}_b^{\gl{a,e}}\colon \prn{\phi_0^\sharp}_b\gl{a,e}\to \prn{\phi_1^\sharp}_b\gl{a,e}$ as follows:
        \begin{align*}
          \prn{\phi^\sharp_{01}}_b^{\gl{a,e}}
          &\colon
          \prn{\phi_0}_a\prn{e}\gl{a,1_a} \to \prn{\phi_1}_a\prn{e}\gl{a,1_a}
          \\
          \prn{\phi^\sharp_{01}}_b^{\gl{a,e}} &\coloneq
          \prn{\phi_{01}}_a\prn{e}\gl{a,1_a}
        \end{align*}

        Naturality of this assignment follows from that of $\phi_{01}$.
      \end{con}

      \bigskip\noindent
      We have defined a functor \[{\sharp_{E,F}}\colon \FIB{A}\prn{E,\OplDiag{p}\prn{F}} \to \FIB{B}\prn{\OplSum{p}\prn{E},F}\text{,}\] setting $\sharp_{E,F}\prn{\phi} :\equiv \phi^\sharp$.
    \end{xsect}

    \begin{xsect}{Pseudo-naturality of transposition}
      The assignment $\prn{E,F}\mapsto \sharp_{E,F}$ can be extended to a pseudo-natural transformation in $\FIB{A}\Op\times \FIB{B}$. Luckily, each of the following pseudo-naturality squares commutes, in fact, on the nose:
      \[
        \DiagramSquare{
          nw = \FIB{A}\prn{E_0,\OplDiag{p}\prn{F_0}},
          ne = \FIB{B}\prn{\OplSum{p}\prn{E_0},F_0},
          sw = \FIB{A}\prn{E_1,\OplDiag{p}\prn{F_1}},
          se = \FIB{B}\prn{\OplSum{p}\prn{E_1},F_1},
          north = \sharp_{E_0,F_0},
          south = \sharp_{E_1,F_1},
          west = \FIB{A}\prn{E_{10},\OplDiag{p}\prn{F_{01}}},
          east = \FIB{B}\prn{\OplSum{p}\prn{E_{10}},F_{01}},
          width = 5cm,
        }
      \]

      This is immediate by unraveling definitions. Thus, the coherence laws for pseudo-naturality data hold as well.
    \end{xsect}

    \begin{xsect}{Pseudo-invertibility of transposition}
      Let $E\in \FIB{A}$ and $F\in\FIB{B}$ be fibred categories.
      In this section, we will show that the transpose functor ${\sharp_{E,F}}\colon \FIB{A}\prn{E,\OplDiag{p}\prn{F}} \to \FIB{B}\prn{\OplSum{p}\prn{E},F}$ is an equivalence of categories. It is worth noting that the pseudo-inverse is \emph{not} a strict inverse; this is the reason why we ultimately get a pseudo-adjunction rather than a 2-adjunction.

      \begin{con}[Action of reverse transposition on objects]
        Given a fibred functor $\phi\colon \OplSum{p}\prn{E}\to F$ over $B$, we define a displayed functor $\phi^\flat\colon E\to\OplDiag{p}\prn{F}$ over $A$ acting as follows:
        \begin{itemize}[label=$\triangleright$]
          \item \emph{On objects.} Given $e\in E$ over $a\in A$, we define a fibred functor $\phi^\flat\prn{e}\colon \OplSum{p}\prn{\Yo{A}\prn{a}}\to F$ as follows:
          \begin{itemize}
            \item \emph{On objects.} Given $\gl{a',f}\in\OplSum{p}\prn{\Yo{A}\prn{a}}$, we define $\phi^\flat\prn{e}\gl{a',f} = \phi\gl{a',\Lift{E}{f}{e}}$.
            \item \emph{On arrows.} Given $\gl{a'_{01},f_{01}}\colon \gl{a'_0,f_0}\to \gl{a'_1,f_1}$ in $\OplSum{p}\prn{\Yo{A}\prn{a}}$ over $b'_{01}\colon b'_0\to b'_1$,\footnote{Here, $f_{01}$ is the ``name'' of the triangle witnessing that $f_1\circ a'_{01}=f_0$.} we define the functorial action $\phi^\flat\prn{e}\gl{a'_{01},f_{01}}\colon \phi\gl{a'_0,\Lift{E}{f_0}{e}}\to \phi\gl{a'_1,\Lift{E}{f_1}{e}}$ to be the functorial action of $\phi$ on the arrow $\gl{a'_{01},\Lift{E}{f_{01}}{e}}\colon \gl{a'_{0},\Lift{E}{f_{0}}{e}}\to \gl{a'_{1},\Lift{E}{f_{1}}{e}}$ where $\Lift{E}{f_{01}}{e}$ is the following canonical comparison map:
            \[
              \begin{tikzpicture}[diagram]
                \SpliceDiagramSquare<l/>{
                  nw = \Lift{E}{f_0}{e},
                  ne = \Lift{E}{f_1}{e},
                  sw = a'_0,
                  se = a'_1,
                  south = a'_{01},
                  north = \Lift{E}{f_{01}}{e},
                  north/style = {->,exists},
                  nw/style = pullback,
                  ne/style = pullback,
                  width = 3cm,
                  height = 1.5cm,
                  north/node/style = upright desc,
                  west/style= {|->},
                  east/style = {|->},
                }
                \SpliceDiagramSquare<r/>{
                  glue = west,
                  glue target = l/,
                  ne = e,
                  se = a,
                  south = f_1,
                  north = \LiftArr{f_1}{e},
                  east/style = {|->},
                  height = 1.5cm,
                  north/node/style = upright desc,
                }
                \draw[->,bend left=30] (l/nw) to node[above] {$\LiftArr{E}{f_0}{e}$}  (r/ne);
              \end{tikzpicture}
            \]

            In other words, we have defined $\phi^\flat\prn{e}\gl{a'_{01},f_{01}} = \phi\prn{\Lift{E}{f_{01}}{e}}$. This definition makes $\phi^\flat\prn{e}\colon \OplSum{p}\prn{\Yo{A}\prn{a}}\to F$ into a functor due to the functoriality of $\phi$ and of the lifting $\Lift{E}{-}{e}$.
            \item \emph{On cartesian arrows.}
              Let $\gl{a'_{01},f_{01}}\colon \gl{a'_0,f_0}\to \gl{a'_1,f_1}$ be cartesian in $\OplSum{p}\prn{\Yo{A}\prn{a}}$, which amounts to the following data:
              \[
                \begin{tikzpicture}[diagram]
                  \SpliceDiagramSquare<n/>{
                    nw = f_0,
                    ne = f_1,
                    north = f_{01},
                    sw = a'_0,
                    se = a'_1,
                    south = a'_{01},
                    nw/style = pullback,
                    sw/style = pullback,
                    south/node/style = upright desc,
                    height = 1.5cm,
                    width = 2.5cm,
                    west/style = {|->},
                    east/style = {|->}
                  }
                  \node (s/sw) [below = 1.5cm of n/sw] {$b'_0$};
                  \node (s/se) [below = 1.5cm of n/se] {$b'_1$};
                  \draw[->] (s/sw) to node[below] {$b'_{01}$} (s/se);
                  \draw[{|->}] (n/sw) to (s/sw);
                  \draw[{|->}] (n/se) to (s/se);
                \end{tikzpicture}
              \]

              To see that $\phi^\flat\prn{e}\gl{a'_{01},f_{01}} = \phi\prn{\Lift{E}{f_{01}}{e}}$ is a cartesian arrow, we recall that $\phi\colon \OplSum{p}\prn{E}\to F$ is fibred, so it suffices to observe that the comparison map $\Lift{E}{f_{01}}{e}\colon \Lift{E}{f_0}{e}\to \Lift{E}{f_1}{e}$ is cartesian by the horizontal pasting lemma for cartesian arrows.
          \end{itemize}

          \item \emph{On arrows.} Given $e_{01}\colon e_0\to e_1$ over $a_{01}\colon a_0\to a_1$, we define $\phi^\flat\prn{e_{01}}\colon \phi^\flat\prn{e_0}\to\phi^\flat\prn{e_1}$ in $\OplDiag{p}\prn{F}$ over $a_{01}$ as follows. Recalling the definition of displayed arrows in $\OplDiag{p}\prn{F}$ we must construct a displayed natural transformation $\phi^\flat\prn{e_{01}}\colon \phi^\flat\prn{e_0}\to \phi^\flat\prn{e_1}\circ \OplSum{p}\prn{\Yo{A}\prn{a_{01}}}$ in $\FIB{B}\prn{\OplSum{p}\prn{\Yo{A}\prn{a_0}},F}$.
          Given $\gl{a',f}\in\OplSum{p}\prn{\Yo{A}\prn{a_0}}$ so that we have $f\colon a'\to a_0$, we consider the following vertical comparison map between cartesian lifts:
          \[
            \begin{tikzpicture}[diagram]
              \node (f*e0) {$\Lift{E}{f}{e_0}$};
              \node[right = 2.5cm of f*e0] (e0) {$e_0$};
              \draw[->] (f*e0) to node[above] {$\LiftArr{E}{f}{e_0}$} (e0);

              \node[pullback, below = of f*e0] (a01f*e1) {$\Lift{E}{\prn{a_{01}\circ f}}{e_1}$};
              \node[right = 5cm of a01f*e1] (e1) {$e_1$};
              \draw[->] (e0) to node[sloped,above] {$e_{01}$} (e1);
              \draw[->] (a01f*e1) to node[upright desc] {$\LiftArr{E}{\prn{a_{01}\circ f}}{e_1}$} (e1);
              \draw[->,exists] (f*e0) to node[left] {$\Lift{E}{f}{e_{01}}$} (a01f*e1);

              \node[below = 1.5cm of a01f*e1] (a') {$a'$};
              \node[right = 2.5cm of a'] (a0) {$a_0$};
              \node[right = 5cm of a'] (a1) {$a_1$};
              \draw[->] (a') to node[below] {$f$} (a0);
              \draw[->] (a0) to node[below] {$a_{01}$} (a1);
              \draw[{|->}] (a01f*e1) to (a');
              \draw[{|->}] (e1) to (a1);
            \end{tikzpicture}
          \]

          Using the above, we define the desired functorial arrow as follows:
          \begin{align*}
            \phi^\flat\prn{e_{01}}\Sub{\gl{a',f}}
            &\colon
            \phi\gl{a',\Lift{E}{f}{e_0}}
            \to
            \phi\gl{a',\Lift{E}{a_{01}\circ f}{e_1}}
            \\
            \phi^\flat\prn{e_{01}}\Sub{\gl{a',f}}
            &=
            \phi\gl{1_{a'}, \Lift{E}{f}{e_{01}}}
          \end{align*}
          This definition makes $\phi^\flat\colon E\to\OplDiag{p}\prn{F}$ into a functor due to the functoriality of the lifting $\Lift{E}{f}{-}$. Moreover, the assignment above is natural.
          \begin{proof}
            To check naturality, fix $\gl{a'_{01},f_{01}}\colon \gl{a'_0,f_0}\to \gl{a'_1,f_1}$ in $\OplSum{p}\prn{\Yo{A}\prn{a}}$ to check that the following square commutes in $F$:
            \[
              \DiagramSquare{
                width = 5cm,
                north = \phi\gl{1_{a'_0}, \Lift{E}{f_0}{e_{01}}},
                nw = \phi\gl{a'_0,\Lift{E}{f_0}{e_0}},
                ne = \phi\gl{a'_0,\Lift{E}{\prn{a_{01}\circ f_0}}{e_1}},
                sw = \phi\gl{a'_1,\Lift{E}{f_1}{e_0}},
                se = \phi\gl{a'_1,\Lift{E}{\prn{a_{01}\circ f_1}}{e_1}},
                west = \phi\gl{a'_{01},\Lift{E}{f_{01}}{e_0}},
                east = \phi\gl{a'_{01},\Lift{E}{f_{01}}{e_1}},
                south = \phi\gl{1_{a'_1}, \Lift{E}{f_1}{e_{01}}},
              }
            \]

            It suffices to check that the following commutes in $E$:
            \[
              \DiagramSquare{
                width = 4cm,
                nw = \Lift{E}{f_0}{e_0},
                ne = \Lift{E}{\prn{a_{01}\circ f_0}}{e_1},
                sw = \Lift{E}{f_1}{e_0},
                se = \Lift{E}{\prn{a_{01}\circ f_1}}{e_1},
                west = \Lift{E}{f_{01}}{e_0},
                north = \Lift{E}{f_0}{e_{01}},
                south = \Lift{E}{f_1}{e_{01}},
                east = \Lift{E}{\prn{a_{01}\cdot f_{01}}}{e_1},
              }
            \]

            By the universal property of $\Lift{E}{a_{01}\circ f}{e_1}$ as a cartesian lift, it suffices to check that the whiskering of the diagram above with $\LiftArr{E}{a_{01}\circ f_1}{e_1}\colon \Lift{E}{a_{01}\circ f_1}{e_1}\to e_1$ commutes. By unraveling definitions and calculating, the upper and lower composites can both be simplified to $e_{01}\circ \LiftArr{E}{f_0}{e_0}\colon \Lift{E}{f_0}{e_0}\to e_1$.
          \end{proof}

          \item \emph{On cartesian arrows.} Now fix a \emph{cartesian} arrow $e_{01}\colon e_0\to e_1$ in $E$  over $a_{01}\colon a_0\to a_1$ in $A$:
          \[
            \DiagramSquare{
              nw/style = pullback,
              nw = e_0,
              ne = e_1,
              sw = a_0,
              se = a_1,
              north = e_{01},
              south = a_{01},
              west/style = {|->},
              east/style = {|->},
              height = 1.5cm,
            }
          \]

          We must check that $\phi^\flat\prn{e_{01}}\colon \phi^\flat\prn{e_0}\to \phi^\flat\prn{e_1}$ is cartesian in $\OplDiag{p}\prn{F}$. This is the same as the underlying natural transformation $\phi^\flat\prn{e_{01}}\colon \phi^\flat\prn{e_0}\to \phi^\flat\prn{e_1}\circ\OplSum{p}\prn{\Yo{A}\prn{a_{01}}}$ being invertible. Fixing $f\colon a'\to a_0$ so that we have $\gl{a',f}\in\OplSum{p}\prn{\Yo{A}\prn{a_0}}$, we must check that the following component in $F$ is invertible:
          \begin{align*}
            \phi^\flat\prn{e_{01}}\Sub{\gl{a',f}}
            &\colon
            \phi\gl{a',\Lift{E}{f}{e_0}}
            \to
            \phi\gl{a',\Lift{E}{a_{01}\circ f}{e_1}}
            \\
            \phi^\flat\prn{e_{01}}\Sub{\gl{a',f}}
            &=
            \phi\gl{1_{a'}, \Lift{E}{f}{e_{01}}}
          \end{align*}

          It is enough to observe that universal gap maps, such as the map $\Lift{E}{f}{e_{01}}\colon \Lift{E}{f}{e_0}\to \Lift{E}{\prn{a_{01}\circ f}}{e_1}$ as defined above, are invertible by the uniqueness of cartesian lifts up to vertical isomorphism.
          \[
            \begin{tikzpicture}[diagram]
              \SpliceDiagramSquare<l/>{
                nw/style = pullback,
                ne/style = pullback,
                nw = \Lift{E}{f}{e_0},
                ne = e_0,
                sw = a',
                se = a_0,
                north = \LiftArr{E}{f}{e_0},
                south = f,
                west/style = {|->},
                east/style = {|->},
                height = 1.5cm,
              }
              \SpliceDiagramSquare<r/>{
                glue = west, glue target = l/,
                ne = e_1,
                se = a_1,
                north = e_{01},
                south = a_{01},
                east/style = {|->},
                height = 1.5cm,
              }
            \end{tikzpicture}
          \]
        \end{itemize}
      \end{con}

      \bigskip

      \begin{xsect}{Transpose is an equivalence}
        We now show that each transpose functor \[ \sharp_{E,F}\colon\FIB{A}\prn{E,\OplDiag{p}\prn{F}}\to \FIB{B}\prn{\OplSum{p}\prn{E},F}\]  is an equivalence.

        \begin{lem}
          The transpose functor $\sharp_{E,F}$ is full.
        \end{lem}

        \begin{proof}
          Fix fibred functors $\phi_0,\phi_1\colon E\to \OplDiag{p}\prn{F}$ over $A$, and a fibred natural transformation $\hat{\phi}_{01} \colon \phi_0^\sharp\to \phi_1^\sharp$. We must find $\phi_{01}\colon \phi_0\to \phi_1$ such that $\phi_{01}^\sharp = \hat{\phi}_{01}$.

          Fixing $e\in E$ over $a$ and $f\colon a'\to a$, we must define $\prn{\phi_{01}^e}_{\gl{a',f}}\colon \phi_0\prn{e}\gl{a',f}\to \phi_1\prn{e}\gl{a',f}$.
          Because each $\phi_i\prn{e}$ is a fibred functor, the following is cartesian as $\gl{a',f}$ is the cartesian lift $\Lift{\OplSum{p}\prn{\Yo{A}\prn{a}}}{f}{\gl{a,1_a}}$.
          \[
            \DiagramSquare{
              nw/style = pullback,
              west/style = {|->},
              east/style = {|->},
              height = 1.5cm,
              width = 4.5cm,
              nw = \phi_i\prn{e}\gl{a',f},
              ne = \phi_i\prn{e}\gl{a,1_a},
              sw = a',
              se = a,
              south = f,
              north = \phi_i\prn{e}\prn{1_{\gl{a',f}}},
            }
          \]

          We already have $\hat{\phi}_{01}^{\gl{a,e}}\colon \phi_0\prn{e}\gl{a,1_a}\to \phi_1\prn{e}\gl{a,1_a}$ by assumption. We can use this together with the universal property of the cartesian arrow depicted above to define $\prn{\phi_{01}^e}_{\gl{a',f}}\colon \phi_0\prn{e}\gl{a',f}\to \phi_1\prn{e}\gl{a',f}$ as follows:
          \[
            \DiagramSquare{
              nw = \phi_0\prn{e}\gl{a',f},
              ne = \phi_0\prn{e}\gl{a,1_a},
              sw = \phi_1\prn{e}\gl{a',f},
              se = \phi_1\prn{e}\gl{a,1_a},
              west/style = {exists, ->},
              width = 5cm,
              north = \phi_0\prn{e}\prn{1_{\gl{a',f}}},
              south = \phi_1\prn{e}\prn{1_{\gl{a',f}}},
              west = \prn{\phi_{01}^e}_{\gl{a',f}},
              east = \hat{\phi}_{01}^{\gl{a,e}},
            }
          \]

          We must check that $\phi_{01}^\sharp = \hat{\phi}_{01}$. Fixing $\gl{a,e}\in \OplSum{p}\prn{E}$ over $b$ so that $e$ lies over $a$ and $p\prn{a}=b$, we must check that $\prn{\phi_{01}^e}_{\gl{a,1_a}} = \hat{\phi}_{01}^{\gl{a,e}}$. By the universal property of the cartesian lift, it suffices to check that the following square commutes:
          \[
            \DiagramSquare{
              width = 5cm,
              nw = \phi_0\prn{e}\gl{a,1_a},
              ne = \phi_0\prn{e}\gl{a,1_a},
              sw = \phi_1\prn{e}\gl{a,1_a},
              se = \phi_1\prn{e}\gl{a,1_a},
              north = \phi_0\prn{e}\prn{1_{\gl{a,1_a}}},
              south = \phi_1\prn{e}\prn{1_{\gl{a,1_a}}},
              west = \hat{\phi}_{01}^{\gl{a,e}},
              east = \hat{\phi}_{01}^{\gl{a,e}}
            }
          \]
          But this is a naturality square for $\hat{\phi}_{01}\colon \phi_0^\sharp\to \phi_1^\sharp$.
        \end{proof}

        \begin{lem}
          The transpose functor ${\sharp_{E,F}}$ is faithful.
        \end{lem}
        \begin{proof}
          We fix a pair of displayed natural transformations $\phi_{01},\psi_{01}\colon\phi_0\to\phi_1$ between fibred functors $\phi_0,\phi_1\colon E\to \OplDiag{p}\prn{F}$ over $A$  such that $\phi_{01}^\sharp = \psi_{01}^\sharp$. We need to check that $\phi_{01} = \psi_{01}$, \ie that for any $e\in E$ over $a\in A$ we have $\phi_{01}\prn{e} = \psi_{01}\prn{e}$ in $\OplDiag{p}\prn{F}$. Fixing $\gl{a',f}\in \OplSum{p}\prn{\Yo{A}\prn{a}}$ over $p\prn{a'}$, we must check that $\phi_{01}\prn{e}\gl{a',f} = \psi_{01}\prn{e}\gl{a',f}$.

          But $\gl{a',f}$ is the cartesian lift $\Lift{\OplSum{p}\prn{\Yo{A}\prn{a}}}{f}{\gl{a,1_a}}$. Because $\phi_0\prn{e}$ and $\phi_1\prn{e}$ are fibred functors, the following are cartesian:
          \[
            \DiagramSquare{
              nw/style = pullback,
              nw = \phi_i\prn{e}\gl{a',f},
              ne = \phi_i\prn{e}\gl{a,1_a},
              sw = a',
              se = a,
              north = \phi_i\prn{e}\prn{1_{\gl{a',f}}},
              south = f,
              west/style = {|->},
              east/style = {|->},
              height = 1.5cm,
              width = 4.5cm,
            }
          \]
          and thus in the following naturality square, $\phi_{01}\prn{e}\gl{a',f}$ is the \emph{unique} vertical map that can be configured like so:
          \[
            \DiagramSquare{
              nw = \phi_0\prn{e}\gl{a',f},
              ne = \phi_0\prn{e}\gl{a,1_a},
              sw = \phi_1\prn{e}\gl{a',f},
              se = \phi_1\prn{e}\gl{a,1_a},
              north = \phi_0\prn{e}\prn{1_{\gl{a',f}}},
              west = \phi_{01}\prn{e}\gl{a',f},
              east = \phi_{01}\prn{e}\gl{a,1_a},
              south = \phi_1\prn{e}\prn{1_{\gl{a',f}}},
              width = 4.5cm,
            }
          \]

          Thus, it suffices to check that the following square commutes:
          \[
            \DiagramSquare{
              nw = \phi_0\prn{e}\gl{a',f},
              ne = \phi_0\prn{e}\gl{a,1_a},
              sw = \phi_1\prn{e}\gl{a',f},
              se = \phi_1\prn{e}\gl{a,1_a},
              north = \phi_0\prn{e}\prn{1_{\gl{a',f}}},
              west = \psi_{01}\prn{e}\gl{a',f},
              east = \phi_{01}\prn{e}\gl{a,1_a},
              south = \phi_1\prn{e}\prn{1_{\gl{a',f}}},
              width = 4.5cm,
            }
          \]

          Because $\phi_{01}^\sharp = \psi_{01}^\sharp$, we have $\phi_{01}\prn{e}\gl{a,1_a} = \psi_{01}\prn{e}\gl{a,1_a}$; thus it suffices to check that the following square commutes:
          \[
            \DiagramSquare{
              nw = \phi_0\prn{e}\gl{a',f},
              ne = \phi_0\prn{e}\gl{a,1_a},
              sw = \phi_1\prn{e}\gl{a',f},
              se = \phi_1\prn{e}\gl{a,1_a},
              north = \phi_0\prn{e}\prn{1_{\gl{a',f}}},
              west = \psi_{01}\prn{e}\gl{a',f},
              east = \psi_{01}\prn{e}\gl{a,1_a},
              south = \phi_1\prn{e}\prn{1_{\gl{a',f}}},
              width = 4.5cm,
            }
          \]
          This is a naturality square for $\psi_{01}\prn{e}$, so we are done.
        \end{proof}

        \begin{lem}
          The reverse transpose operation exhibits the transpose functor $\sharp_{E,F}$ as \emph{(split) essentially surjective}.
        \end{lem}

        \begin{proof}
          Fixing a fibred functor $\phi\colon\OplSum{p}\prn{E}\to F$ over $B$, we will check that $\phi^{\flat\sharp}$ is naturally isomorphic to $\phi$. By definition, we have the following computation of the component of $\phi^{\flat\sharp}$ at $\gl{a,e}\in \OplSum{p}\prn{E}$ over $b\in B$:
          \[
            \phi^{\flat\sharp}_b\gl{a,e} = \phi^\flat_a\prn{e}\gl{a,1_a}
            = \phi_b\gl{a,\Lift{E}{1_a}{e}}
          \]

          The cartesian arrow $\LiftArr{E}{1_a}{e}\colon\Lift{E}{1_a}{e}\to e$ is an isomorphism because $1_a$ is an identity arrow, so we have an isomorphism $\phi_{1_b}\gl{1_a,\LiftArr{E}{1_a}{e}}\colon \phi^{\flat\sharp}_b\gl{a,e}\to \phi_b\gl{a,e}$ by functoriality. For naturality, we fix $e_{01}\colon e_0\to e_1$ over $a_{01}\colon a_0\to a_1$ over $b_{01}\colon b_0\to b_1$ to check that the following commutes in $F$:
          \[
            \DiagramSquare{
              width = 4.5cm,
              nw = \phi^{\flat\sharp}_{b_0}\gl{a_0,e_0},
              ne = \phi^{\flat\sharp}_{b_1}\gl{a_1,e_1},
              sw = \phi_{b_0}\gl{a_0,e_0},
              se = \phi_{b_1}\gl{a_1,e_1},
              north = \phi^{\flat\sharp}_{b_{01}}\gl{a_{01},e_{01}},
              west = \phi_{1_{b_0}}\gl{1_{a_0},\LiftArr{E}{1_{a_0}}{e_0}},
              east = \phi_{1_{b_1}}\gl{1_{a_1},\LiftArr{E}{1_{a_1}}{e_1}},
              south = \phi_{b_{01}}\gl{a_{01},e_{01}},
            }
          \]

          We begin with the following commuting square in $E$:
          \[
            \DiagramSquare{
              nw = \Lift{E}{1_{a_0}}{e_0},
              ne = \Lift{E}{a_{01}}{e_1},
              sw = e_0,
              se = e_1,
              north = \Lift{E}{1_{a_0}}{e_{01}},
              west = \LiftArr{E}{1_{a_0}}{e_0},
              east = \LiftArr{E}{a_{01}}{e_1},
              south = e_{01},
              width = 2.5cm,
            }
          \]

          Under $\phi\colon\OplSum{p}\prn{E}\to F$, we obtain the following:
          \[
            \DiagramSquare{
              nw = \phi_{b_0}\gl{a_0,\Lift{E}{1_{a_0}}{e_0}},
              ne = \phi_{b_0}\gl{a_0,\Lift{E}{a_{01}}{e_1}},
              sw = \phi_{b_0}\gl{a_0,e_0},
              se = \phi_{b_1}\gl{a_1,e_1},
              north = \phi_{1_{b_0}}\gl{1_{a_0},\Lift{E}{1_{a_0}}{e_{01}}},
              west = \phi_{1_{b_0}}\gl{1_{a_0},\LiftArr{E}{1_{a_0}}{e_0}},
              east = \phi_{b_{01}}\gl{a_{01},\LiftArr{E}{a_{01}}{e_1}},
              south = \phi_{b_{01}}\gl{a_{01},e_{01}},
              width = 5cm,
            }
          \]

          By definition of the comparison map $\Lift{E}{!\Sub{\gl{a_0,a_{01}}}}{e_0}$, we have the following triangle that factors the eastern map in the square above:
          \[
            \begin{tikzpicture}[diagram]
              \node (nw) {$\phi_{b_0}\gl{a_0,\Lift{E}{a_{01}}{e_1}}$};
              \node[below right = 4.5cm of nw] (s) {$\phi_{b_1}\gl{a_1,e_1}$};
              \node[above right = 4.5cm of s] (ne) {$\phi_{b_1}\gl{a_1,\Lift{E}{1_{a_1}}{e_1}}$};
              \draw[->] (nw) to node[above] {$\phi_{b_{01}}\gl{a_{01},\Lift{E}{!\Sub{\gl{a_0,a_{01}}}}{e_0}}$} (ne);
              \draw[->] (nw) to node[sloped,below] {$\phi_{b_{01}}\gl{a_{01},\LiftArr{E}{a_{01}}{e_1}}$} (s);
              \draw[->] (ne) to node[sloped,below] {$\phi_{1_{b_1}}\gl{1_{a_1},\LiftArr{E}{1_{a_1}}{e_1}}$} (s);
            \end{tikzpicture}
          \]

          Returning to the original naturality square that we wished to check, we unfold definitions to see that $\phi^{\flat\sharp}_{b_{01}}\gl{a_{01},e_{01}}$ is the composite $\phi_{b_{01}}\gl{a_{01},\Lift{E}{!\Sub{\gl{a_0,a_{01}}}}{e_0}}\circ \phi_{1_{b_0}}\gl{1_{a_0},\Lift{E}{1_{a_0}}{e_{01}}}$. Therefore, we have:
          \begin{align*}
            &\phi_{1_{b_1}}\gl{1_{a_1},\LiftArr{E}{1_{a_1}}{e_1}}\circ
            \phi^{\flat\sharp}_{b_{01}}\gl{a_{01},e_{01}}
            \\
            &\quad=
            \phi_{1_{b_1}}\gl{1_{a_1},\LiftArr{E}{1_{a_1}}{e_1}}\circ
            \phi_{b_{01}}\gl{a_{01},\Lift{E}{!\Sub{\gl{a_0,a_{01}}}}{e_0}}
            \circ\phi_{1_{b_0}}\gl{1_{a_0},\Lift{E}{1_{a_0}}{e_{01}}}
            \\
            &\quad=
            \phi_{b_{01}}\gl{a_{01},\LiftArr{E}{a_{01}}{e_1}}
            \circ \phi_{1_{b_0}}\gl{1_{a_0},\Lift{E}{1_{a_0}}{e_{01}}}
            \\
            &\quad=
            \phi_{b_{01}}\gl{a_{01},e_{01}}\circ \phi_{1_{b_0}}\gl{1_{a_0},\LiftArr{E}{1_{a_0}}{e_0}}
            \qedhere
          \end{align*}
        \end{proof}

        \begin{cor}
          The transpose functor $\sharp_{E,F}$ is a (strong) equivalence of categories.
        \end{cor}

      \end{xsect}

      \begin{cor}
        For a fibration $p\colon A\rightarrowtriangle B$, the oplax sum pseudofunctor $\OplSum{p}\colon\FIB{A}\to\FIB{B}$ has a right pseudo-adjoint given by $\OplDiag{p}\colon \FIB{B}\to\FIB{A}$.
      \end{cor}
    \end{xsect}

  \end{xsect}
\end{xsect}

\begin{xsect}[sec:relative-hs-lifting]{Relative Hofmann--Streicher lifting}
  The (relative, generalised) Hofmann--Streicher lifting of a fibred category $E\in \FIB{B}$ along a fibration $p\colon A\rightarrowtriangle B$ is finally obtained by applying the \emph{lax} base change functor $\LaxDiag{p}\colon\FIB{B}\to\FIB{A}$ to $E$ as described in \S~\ref{sec:intro:lax-base-change}:
  \[
    \LaxDiag{p}\prn{E} =
    \OplDiag{p}\prn{E\Op}\Op
  \]

  As we have explained, the generalisation from ordinary categories $A$ to fibrations $p\colon A\rightarrowtriangle B$ allows us to iterate the Hofmann--Streicher lifting construction along a sequence of fibrations $A_0\rightarrowtriangle \cdots A_n\rightarrowtriangle B$; post-composing finally with $B\rightarrowtriangle\mathbf{1}$ we see that our construction agrees with the usual ``absolute'' Hofmann--Streicher lifting of a category.
\end{xsect}

\begin{xsect}{A new 2-dimensional bifibration of fibrations}
    The relationship between fibred categories and their base categories is well-expressed by the 2-fibration $\FIB{}\to\CAT$ in which cartesian lifts of one-cells are given by pullbacks of fibrations. Hermida~\cite{hermida:2004} explains that for a functor $f\colon A\to B$, each of the pullback 2-functors $f^*\colon \FIB{B}\to \FIB{A}$ has a left pseudo-adjoint $f_!\colon \FIB{A}\to \FIB{B}$; the full computation of the left pseudo-adjoint is not easy, but its fibre categories are described by the following pseudo-colimits:
    \[
      \prn{f_!E}_b =
      \mathrm{pscolim}\prn{
        \DelimMin{1}
        \prn{b/f}\Op
        \xrightarrow{\partial_1\Op}
        A\Op
        \xrightarrow{E_\bullet}
        \CAT
      }
    \]

    In summary, these left adjoints endow the 2-fibration $\FIB{}\rightarrowtriangle\CAT$ with \emph{direct images} in the sense of Hermida~\cite{hermida:2004}, making $\FIB{}\rightarrowtriangle\CAT$ into some kind of 2-dimensional bifibration---although there seems to be some uncertainty in the literature as to the proper definition of bifibrations of 2-categories.\footnote{Even the definition of ``2-fibration'' is contested in the literature: see Buckley~\cite[Remark~2.1.9]{buckley:2014} for discussion. Buckley's definition is strictly stronger than that of Hermida~\cite{hermida:1993:fib}, and this strength is necessary in order to obtain the straightening/unstraightening theory of 2-fibrations as Buckley points out, apparently correcting a claim of Bakovi\'c~\cite{bakovic:2011}. Therefore, any work that explicitly depends on the elegant results of Hermida~\cite{hermida:1993:fib,hermida:2004} on 2-fibrations would need to check explicitly whether these results are in fact compatible with Buckley's corrected definition of 2-fibrations.}

    We have shown that for each fibration $p\colon A\rightarrowtriangle B$ we have a pseudo-adjunction $\OplSum{p}\dashv\OplDiag{p}\colon\FIB{A}\to\FIB{B}$, and it is natural to ask whether these play a part in a \emph{different} 2-dimensional bifibration. Naturally, the base 2-category cannot be $\CAT$, but must instead be a 2-category whose 1-cells are fibrations.

    \begin{defi}
      We define $\CATFib$ to be the 2-category of categories $C$, fibrations $p\colon A\rightarrowtriangle B$, and natural transformations $\alpha\colon p'\Rightarrow p$. In other words, this is the (non-full) sub-2-category of $\CAT\Co$ with 1-cells restricted to fibrations.
    \end{defi}

    \begin{defi}
      Now define $\FIBFib$ to be the 2-category specified as follows:
      \begin{itemize}[label=$\triangleright$]
        \item A 0-cell of $\FIBFib$ is a fibration $p\colon A\rightarrowtriangle B$.
        \item A 1-cell of $\FIBFib$ from $p\colon A\rightarrowtriangle B$ to $q\colon C\rightarrowtriangle D$ is given by a square
        \[
          \begin{tikzpicture}[scale=0.5,baseline=($(nw)!0.5!(se)$)]
            \CreateRect{6}{4}
              \path
                coordinate[label=above:$\strut p$] (nw) at (spath cs:north 0.25)
                coordinate[label=above:$\strut g$] (ne) at (spath cs:north 0.75)
                coordinate[label=below:$\strut f$] (sw) at (spath cs:south 0.25)
                coordinate[label=below:$\strut q$] (se) at (spath cs:south 0.75)
              ;
              \draw[spath/save = pq] (nw) to[out=-90,in=90] (se);
              \draw[spath/save = gf, exists] (ne) to[out=-90,in=90] (sw);
              \path[name intersections={of=pq and gf}]
                coordinate[dot,label=90:$\alpha$] (cell) at (intersection-1)
              ;
          \end{tikzpicture}
        \]
        in which $g\colon B\rightarrowtriangle D$ is a fibration.
        \item Given a pair of 1-cells as above,
        \[
          \begin{tikzpicture}[scale=0.5,baseline=($(nw)!0.5!(se)$)]
            \CreateRect{6}{4}
              \path
                coordinate[label=above:$\strut p$] (nw) at (spath cs:north 0.25)
                coordinate[label=above:$\strut g$] (ne) at (spath cs:north 0.75)
                coordinate[label=below:$\strut f$] (sw) at (spath cs:south 0.25)
                coordinate[label=below:$\strut q$] (se) at (spath cs:south 0.75)
              ;
              \draw[spath/save = pq] (nw) to[out=-90,in=90] (se);
              \draw[spath/save = gf, exists] (ne) to[out=-90,in=90] (sw);
              \path[name intersections={of=pq and gf}]
                coordinate[dot,label=90:$\alpha$] (cell) at (intersection-1)
              ;
          \end{tikzpicture}
          \qquad
          \begin{tikzpicture}[scale=0.5,baseline=($(nw)!0.5!(se)$)]
            \CreateRect{6}{4}
              \path
                coordinate[label=above:$\strut p$] (nw) at (spath cs:north 0.25)
                coordinate[label=above:$\strut k$] (ne) at (spath cs:north 0.75)
                coordinate[label=below:$\strut h$] (sw) at (spath cs:south 0.25)
                coordinate[label=below:$\strut q$] (se) at (spath cs:south 0.75)
              ;
              \draw[spath/save = pq] (nw) to[out=-90,in=90] (se);
              \draw[spath/save = kh, exists] (ne) to[out=-90,in=90] (sw);
              \path[name intersections={of=pq and kh}]
                coordinate[dot,label=90:$\beta$] (cell) at (intersection-1)
              ;
          \end{tikzpicture}
        \]
        a 2-cell from $\alpha$ to $\beta$ is a pair of natural transformations $\phi\colon f \Rightarrow h$ and $\psi\colon k\Rightarrow g$ making the following equation hold:
        \[
         \beta =
          \begin{tikzpicture}[scale=0.5,baseline=($(nw)!0.5!(se)$)]
            \CreateRect{6}{4}
              \path
                coordinate[label=above:$\strut p$] (nw) at (spath cs:north 0.25)
                coordinate[label=above:$\strut k$] (ne) at (spath cs:north 0.75)
                coordinate[label=below:$\strut h$] (sw) at (spath cs:south 0.25)
                coordinate[label=below:$\strut q$] (se) at (spath cs:south 0.75)
              ;
              \draw[spath/save = pq] (nw) to[out=-90,in=90] (se);
              \draw[spath/save = kh] (ne) to[out=-90,in=90] (sw);
              \path
                coordinate[dot,label=left:$\phi$] (phi) at (spath cs:kh 0.75)
                coordinate[dot,label=right:$\psi$] (psi) at (spath cs:kh 0.25)
              ;
              \path[name intersections={of=pq and kh}]
                coordinate[dot,label=90:$\alpha$] (cell) at (intersection-1)
              ;
          \end{tikzpicture}
        \]

        \item The definitions of identities and compositions of 1- and 2-cells and their strict functoriality, associativity and unit laws are evident.
      \end{itemize}
    \end{defi}

    \begin{con}
      We may now consider the codomain 2-functor $\partial_1 \colon \FIBFib\to \CATFib$.
      \begin{itemize}[label=$\triangleright$]
        \item Given a fibration $p\colon A\rightarrowtriangle B$, we define $\partial_1 p :\equiv B$.
        \item Given fibrations $p\colon A\rightarrowtriangle B$ and $q\colon C\rightarrowtriangle D$, we send a square
          \[
            \begin{tikzpicture}[scale=0.5,baseline=($(nw)!0.5!(se)$)]
              \CreateRect{6}{4}
                \path
                  coordinate[label=above:$\strut p$] (nw) at (spath cs:north 0.25)
                  coordinate[label=above:$\strut g$] (ne) at (spath cs:north 0.75)
                  coordinate[label=below:$\strut f$] (sw) at (spath cs:south 0.25)
                  coordinate[label=below:$\strut q$] (se) at (spath cs:south 0.75)
                ;
                \draw[spath/save = pq] (nw) to[out=-90,in=90] (se);
                \draw[spath/save = gf, exists] (ne) to[out=-90,in=90] (sw);
                \path[name intersections={of=pq and gf}]
                  coordinate[dot,label=90:$\alpha$] (cell) at (intersection-1)
                ;
            \end{tikzpicture}
          \]
          to the fibration $\partial_q\alpha :\equiv g\colon B\rightarrowtriangle D$.
        \item Given a pair of squares
        \[
          \begin{tikzpicture}[scale=0.5,baseline=($(nw)!0.5!(se)$)]
            \CreateRect{6}{4}
              \path
                coordinate[label=above:$\strut p$] (nw) at (spath cs:north 0.25)
                coordinate[label=above:$\strut g$] (ne) at (spath cs:north 0.75)
                coordinate[label=below:$\strut f$] (sw) at (spath cs:south 0.25)
                coordinate[label=below:$\strut q$] (se) at (spath cs:south 0.75)
              ;
              \draw[spath/save = pq] (nw) to[out=-90,in=90] (se);
              \draw[spath/save = gf, exists] (ne) to[out=-90,in=90] (sw);
              \path[name intersections={of=pq and gf}]
                coordinate[dot,label=90:$\alpha$] (cell) at (intersection-1)
              ;
          \end{tikzpicture}
          \qquad
          \begin{tikzpicture}[scale=0.5,baseline=($(nw)!0.5!(se)$)]
            \CreateRect{6}{4}
              \path
                coordinate[label=above:$\strut p$] (nw) at (spath cs:north 0.25)
                coordinate[label=above:$\strut k$] (ne) at (spath cs:north 0.75)
                coordinate[label=below:$\strut h$] (sw) at (spath cs:south 0.25)
                coordinate[label=below:$\strut q$] (se) at (spath cs:south 0.75)
              ;
              \draw[spath/save = pq] (nw) to[out=-90,in=90] (se);
              \draw[spath/save = kh, exists] (ne) to[out=-90,in=90] (sw);
              \path[name intersections={of=pq and kh}]
                coordinate[dot,label=90:$\beta$] (cell) at (intersection-1)
              ;
          \end{tikzpicture}
        \]  with $g,k\colon B\rightarrowtriangle D$ fibrations
        and 2-cell $\gl{\phi,\psi} \colon \alpha\Rightarrow \beta$ as shown below
        \[
         \beta =
          \begin{tikzpicture}[scale=0.5,baseline=($(nw)!0.5!(se)$)]
            \CreateRect{6}{4}
              \path
                coordinate[label=above:$\strut p$] (nw) at (spath cs:north 0.25)
                coordinate[label=above:$\strut k$] (ne) at (spath cs:north 0.75)
                coordinate[label=below:$\strut h$] (sw) at (spath cs:south 0.25)
                coordinate[label=below:$\strut q$] (se) at (spath cs:south 0.75)
              ;
              \draw[spath/save = pq] (nw) to[out=-90,in=90] (se);
              \draw[spath/save = kh] (ne) to[out=-90,in=90] (sw);
              \path
                coordinate[dot,label=left:$\phi$] (phi) at (spath cs:kh 0.75)
                coordinate[dot,label=right:$\psi$] (psi) at (spath cs:kh 0.25)
              ;
              \path[name intersections={of=pq and kh}]
                coordinate[dot,label=90:$\alpha$] (cell) at (intersection-1)
              ;
          \end{tikzpicture}
        \]
        we define $\partial_1\gl{\phi,\psi}\colon \partial_1\alpha \Rightarrow \partial_1\beta$ to be the natural transformation $\psi\colon k\Rightarrow g$, which is, taking into account the local duality involution, a 2-cell $g\to k$ in $\CATFib$.
        \item These assignments are seen to be strictly 2-functorial.
      \end{itemize}
    \end{con}

    \begin{lem}
      The codomain 2-functor $\partial_1\colon \FIBFib\to \CATFib$ is closed under 1-dimensional opcartesian lifts in the sense of Buckley~\cite{buckley:2014}, which are supplied fibrewise by the postcomposition 2-functors $\OplSum{f}\colon \FIB{B}\to \FIB{C}$ for fibrations $f\colon B\rightarrowtriangle C$.
    \end{lem}

    \begin{proof}
      Let $p\colon A\rightarrowtriangle B$ and $f\colon B\rightarrowtriangle C$ be fibrations, as depicted in the following internal diagram of $\partial_1\colon \FIBFib\to \CATFib$:
      \[
        \begin{tikzpicture}[diagram]
          \node (p) {$p$};
          \node[below = 1.5cm of p] (B) {$B$};
          \node[right = of B] (C) {$C$};
          \draw[|->] (p) to (B);
          \draw[fibration] (B) to node[below] {$f$} (C);
          \node[right=of C] (CATFib) {$\CATFib$};
          \node[above=1.5cm of CATFib] (FIBFib) {$\FIBFib$};
          \draw[->] (FIBFib) to node[right] {$\partial_1$} (CATFib);
        \end{tikzpicture}
      \]

      We have the following strictly commuting square in $\CAT$
      \[
        \DiagramSquare{
          nw = A,
          sw = B,
          ne = A,
          se = C,
          north/style = double,
          west = p,
          east = f\circ p,
          south = f,
          west/style = fibration,
          south/style = fibration,
          east/style = fibration,
        }
      \]
      which we shall draw as the following internal diagram of $\partial_1\colon\FIBFib\to\CATFib$
      \[
        \DiagramSquare{
          nw = p,
          sw = B,
          north = \bar{f},
          ne = \OplSum{f}\prn{p},
          se = C,
          height = 1.5cm,
          west/style = {|->},
          east/style = {|->},
          south/style = fibration,
          south = f,
        }
      \]
      which we claim to be opcartesian in the sense of Buckley~\cite{buckley:2014}, which is to say that the following diagram is a (strict) pullback in $\CAT$ for any fibration $q\colon D\rightarrowtriangle E$:
      \[
        \DiagramSquare{
          width = 4cm,
          nw/style = dotted pullback,
          nw = \FIBFib\prn{\OplSum{f}\prn{p},q},
          ne = \FIBFib\prn{p,q},
          sw = \CATFib\prn{C,E},
          se = \CATFib\prn{B,E}
        }
      \]

      We shall check that $\FIBFib\prn{\OplSum{f}\prn{p},q}$ is \emph{isomorphic} to the indicated strict fibre product of categories. First of all, a 1-cell $\OplSum{f}(p)\to q$ in $\FIBFib$ is given by a cell of the following kind in which $k\colon C\rightarrowtriangle E$ is a fibration and $h\colon A\to D$ is a functor:
      \[
        \begin{tikzpicture}[scale=0.5,baseline=($(nw)!0.5!(se)$)]
          \CreateRect{6}{4}
            \path
              coordinate[label=above:$\strut p$] (nw) at (spath cs:north 0.25)
              coordinate[label=above:$\strut k$] (ne) at (spath cs:north 0.75)
              coordinate[label=above:$\strut f$] (n) at (spath cs:north 0.5)
              coordinate[label=below:$\strut h$] (sw) at (spath cs:south 0.25)
              coordinate[label=below:$\strut q$] (se) at (spath cs:south 0.75)
            ;
            \draw[spath/save = dr] (nw) to[out=-90,in=90] (se);
            \draw[spath/save = dl, exists] (ne) to[out=-90,in=90] (sw);
            \path[name intersections={of=dr and dl}]
              coordinate[dot,label=below:$\alpha$] (cell) at (intersection-1)
            ;
            \draw (n) to (cell.center);
        \end{tikzpicture}
      \]
      Conversely, we may take a fibration $k\colon C\rightarrowtriangle E$ and \emph{precompose} with $f\colon B\rightarrowtriangle C$ to obtain $f;k\colon B\rightarrowtriangle E$, and if we consider the class of 1-cells $p\to q$ in $\FIBFib$ lying over it, we obtain precisely the collection of 1-cells $\OplSum{f}\prn{p}\to q$ in $\FIBFib$ lying over $k\colon C\rightarrowtriangle E$.
      It is a similar story to establish the isomorphism on 2-cells.
    \end{proof}

    \begin{lem}
      The codomain 2-functor $\partial_1\colon \FIBFib\to \CATFib$ is \emph{locally} an opfibration, in the sense that for any fibrations $p\colon A\rightarrowtriangle B$ and $q\colon C\rightarrowtriangle D$, the induced functor $\FIBFib\prn{p,q}\to \CATFib\prn{B,C}$ is a opfibration.
    \end{lem}

    \begin{proof}
      We consider an opcartesian lifting problem in $\FIBFib\prn{p,q}\to \CATFib\prn{B,C}$. This amounts to a 1-cell $\gl{f,g,\alpha}\colon p\to q$ in $\FIBFib$ with $g\colon B\rightarrowtriangle D$ a fibration and $f\colon A\to C$ a functor as depicted below
      \[
        \begin{tikzpicture}[scale=0.5,baseline=($(nw)!0.5!(se)$)]
          \CreateRect{6}{4}
            \path
              coordinate[label=above:$\strut p$] (nw) at (spath cs:north 0.25)
              coordinate[label=above:$\strut g$] (ne) at (spath cs:north 0.75)
              coordinate[label=below:$\strut f$] (sw) at (spath cs:south 0.25)
              coordinate[label=below:$\strut q$] (se) at (spath cs:south 0.75)
            ;
            \draw[spath/save = pq] (nw) to[out=-90,in=90] (se);
            \draw[spath/save = gf, exists] (ne) to[out=-90,in=90] (sw);
            \path[name intersections={of=pq and gf}]
              coordinate[dot,label=90:$\alpha$] (cell) at (intersection-1)
            ;
        \end{tikzpicture}
      \]
      together with a fibration $g'\colon B\rightarrowtriangle D$ and a natural transformation $\gamma\colon g'\Rightarrow g$. Together, all this amounts to the following internal diagram in $\FIBFib\prn{p,q}\to \CATFib\prn{B,C}$, taking note of the local duality involution:
      \[
        \begin{tikzpicture}[diagram]
          \node (alpha) {$\alpha$};
          \node[below = 1.5cm of alpha] (g) {$g$};
          \node[right = of g] (g') {$g'$};
          \draw[|->] (alpha) to (g);
          \draw[->] (g) to node[below] {$\gamma$} (g');
          \node[right=of g'] (CATFib) {$\CATFib\prn{B,C}$};
          \node[above=1.5cm of CATFib] (FIBFib) {$\FIBFib\prn{p,q}$};
          \draw[->] (FIBFib) to (CATFib);
        \end{tikzpicture}
      \]

      We may obtain the following morphism from $\gamma_!\alpha \colon p \to q$ over $g'$:
      \[
        \gamma_!\alpha :\equiv
        \begin{tikzpicture}[scale=0.5,baseline=($(nw)!0.5!(se)$)]
          \CreateRect{6}{4}
            \path
              coordinate[label=above:$\strut p$] (nw) at (spath cs:north 0.25)
              coordinate[label=above:$\strut g'$] (ne) at (spath cs:north 0.75)
              coordinate[label=below:$\strut f$] (sw) at (spath cs:south 0.25)
              coordinate[label=below:$\strut q$] (se) at (spath cs:south 0.75)
            ;
            \draw[spath/save = pq] (nw) to[out=-90,in=90] (se);
            \draw[spath/save = gf] (ne) to[out=-90,in=90] (sw);
            \path coordinate[dot,label=right:$\gamma$] (gamma) at (spath cs:gf 0.25);
            \path[name intersections={of=pq and gf}]
              coordinate[dot,label=below:$\alpha$] (cell) at (intersection-1)
            ;
        \end{tikzpicture}
      \]

      There is an evident 2-cell $\gl{1_f,\gamma}\colon \alpha\to \gamma_!\alpha$ over $\gamma$ as witnessed by the unit law:
      \[
        \gamma_!\alpha \equiv
        \begin{tikzpicture}[scale=0.5,baseline=($(nw)!0.5!(se)$)]
          \CreateRect{6}{4}
            \path
              coordinate[label=above:$\strut p$] (nw) at (spath cs:north 0.25)
              coordinate[label=above:$\strut g'$] (ne) at (spath cs:north 0.75)
              coordinate[label=below:$\strut f$] (sw) at (spath cs:south 0.25)
              coordinate[label=below:$\strut q$] (se) at (spath cs:south 0.75)
            ;
            \draw[spath/save = pq] (nw) to[out=-90,in=90] (se);
            \draw[spath/save = gf] (ne) to[out=-90,in=90] (sw);
            \path coordinate[dot,label=right:$\gamma$] (gamma) at (spath cs:gf 0.25);
            \path[name intersections={of=pq and gf}]
              coordinate[dot,label=below:$\alpha$] (cell) at (intersection-1)
            ;
        \end{tikzpicture}
        =
        \begin{tikzpicture}[scale=0.5,baseline=($(nw)!0.5!(se)$)]
          \CreateRect{6}{4}
            \path
              coordinate[label=above:$\strut p$] (nw) at (spath cs:north 0.25)
              coordinate[label=above:$\strut g'$] (ne) at (spath cs:north 0.75)
              coordinate[label=below:$\strut f$] (sw) at (spath cs:south 0.25)
              coordinate[label=below:$\strut q$] (se) at (spath cs:south 0.75)
            ;
            \draw[spath/save = pq] (nw) to[out=-90,in=90] (se);
            \draw[spath/save = kh] (ne) to[out=-90,in=90] (sw);
            \path
              coordinate[dot,label=left:$1_f$] (phi) at (spath cs:kh 0.75)
              coordinate[dot,label=right:$\gamma$] (psi) at (spath cs:kh 0.25)
            ;
            \path[name intersections={of=pq and kh}]
              coordinate[dot,label=90:$\alpha$] (cell) at (intersection-1)
            ;
        \end{tikzpicture}
      \]

      Hence we have the following internal diagram in $\FIBFib\prn{p,q}\to \CATFib\prn{B,C}$:
      \[
        \begin{tikzpicture}[diagram]
          \node (alpha) {$\alpha$};
          \node[right = of alpha] (alpha') {$\gamma_!\alpha$};
          \node[below = 1.5cm of alpha] (g) {$g$};
          \node[right = of g] (g') {$g'$};
          \draw[|->] (alpha) to (g);
          \draw[|->] (alpha') to (g');
          \draw[->] (alpha) to node[above] {$\gl{1_f,\gamma}$} (alpha');
          \draw[->] (g) to node[below] {$\gamma$} (g');
          \node[right=of g'] (CATFib) {$\CATFib\prn{B,C}$};
          \node[above=1.5cm of CATFib] (FIBFib) {$\FIBFib\prn{p,q}$};
          \draw[->] (FIBFib) to (CATFib);
        \end{tikzpicture}
      \]

      We claim that the diagram above is opcartesian. Fixing a 1-cell $\beta \colon p\to q$ in $\FIBFib$
      \[
        \begin{tikzpicture}[scale=0.5,baseline=($(nw)!0.5!(se)$)]
          \CreateRect{6}{4}
            \path
              coordinate[label=above:$\strut p$] (nw) at (spath cs:north 0.25)
              coordinate[label=above:$\strut k$] (ne) at (spath cs:north 0.75)
              coordinate[label=below:$\strut h$] (sw) at (spath cs:south 0.25)
              coordinate[label=below:$\strut q$] (se) at (spath cs:south 0.75)
            ;
            \draw[spath/save = pq] (nw) to[out=-90,in=90] (se);
            \draw[spath/save = kh, exists] (ne) to[out=-90,in=90] (sw);
            \path[name intersections={of=pq and kh}]
              coordinate[dot,label=90:$\beta$] (cell) at (intersection-1)
            ;
        \end{tikzpicture}
      \]
      we must check that the following evident square is cartesian:
      \[
        \DiagramSquare{
          width = 5cm,
          nw = \FIBFib\prn{p,q}\prn{\gamma_!\alpha,\beta},
          ne = \FIBFib\prn{p,q}\prn{\alpha,\beta},
          sw = \CATFib\prn{p,q}\prn{g',k},
          se = \CATFib\prn{p,q}\prn{g,k},
          nw/style = dotted pullback,
        }
      \]

      The actual fibre product consists of transformations $\phi\colon h\Rightarrow f$ and $\psi\colon k\Rightarrow g$ and $\chi\colon k\Rightarrow g'$ satisfying the equation
      \[
         \beta =
          \begin{tikzpicture}[scale=0.5,baseline=($(nw)!0.5!(se)$)]
            \CreateRect{6}{4}
              \path
                coordinate[label=above:$\strut p$] (nw) at (spath cs:north 0.25)
                coordinate[label=above:$\strut k$] (ne) at (spath cs:north 0.75)
                coordinate[label=below:$\strut h$] (sw) at (spath cs:south 0.25)
                coordinate[label=below:$\strut q$] (se) at (spath cs:south 0.75)
              ;
              \draw[spath/save = pq] (nw) to[out=-90,in=90] (se);
              \draw[spath/save = kh] (ne) to[out=-90,in=90] (sw);
              \path
                coordinate[dot,label=left:$\phi$] (phi) at (spath cs:kh 0.75)
                coordinate[dot,label=right:$\psi$] (psi) at (spath cs:kh 0.25)
              ;
              \path[name intersections={of=pq and kh}]
                coordinate[dot,label=90:$\alpha$] (cell) at (intersection-1)
              ;
          \end{tikzpicture}
        \]
        such that $\psi = k \xRightarrow{\chi}g' \xRightarrow{\gamma} g$. But this is precisely the data of a 2-cell $\gamma_!\alpha\to \beta$ in $\FIBFib$.
    \end{proof}

    \begin{lem}
      Local opcartesian 2-cells in $\partial_1\colon\FIBFib\to\CATFib$ are closed under horizontal composition.
    \end{lem}

    \begin{proof}
      Fix fibrations $p\colon A\rightarrowtriangle B$ and $q\colon C\rightarrowtriangle D$ and $r\colon E\rightarrowtriangle F$, and consider 1-cells $\alpha\colon p\to q$ and $\beta\colon q\to r$ in $\FIBFib$ as depicted below:
      \[
        \begin{tikzpicture}[scale=0.5,baseline=($(nw)!0.5!(se)$)]
          \CreateRect{6}{4}
            \path
              coordinate[label=above:$\strut p$] (nw) at (spath cs:north 0.25)
              coordinate[label=above:$\strut g$] (ne) at (spath cs:north 0.75)
              coordinate[label=below:$\strut f$] (sw) at (spath cs:south 0.25)
              coordinate[label=below:$\strut q$] (se) at (spath cs:south 0.75)
            ;
            \draw[spath/save = pq] (nw) to[out=-90,in=90] (se);
            \draw[spath/save = gf, exists] (ne) to[out=-90,in=90] (sw);
            \path[name intersections={of=pq and gf}]
              coordinate[dot,label=90:$\alpha$] (cell) at (intersection-1)
            ;
        \end{tikzpicture}
        \quad
        \begin{tikzpicture}[scale=0.5,baseline=($(nw)!0.5!(se)$)]
          \CreateRect{6}{4}
            \path
              coordinate[label=above:$\strut q$] (nw) at (spath cs:north 0.25)
              coordinate[label=above:$\strut k$] (ne) at (spath cs:north 0.75)
              coordinate[label=below:$\strut h$] (sw) at (spath cs:south 0.25)
              coordinate[label=below:$\strut r$] (se) at (spath cs:south 0.75)
            ;
            \draw[spath/save = qr] (nw) to[out=-90,in=90] (se);
            \draw[spath/save = kh, exists] (ne) to[out=-90,in=90] (sw);
            \path[name intersections={of=qr and kh}]
              coordinate[dot,label=90:$\beta$] (cell) at (intersection-1)
            ;
        \end{tikzpicture}
      \]

      We may consider the following horizontal composition $\alpha * \beta \colon p \to r$ in $\FIBFib$:
      \[
        \begin{tikzpicture}
          \CreateRect{6}{4}
          \path
            coordinate[label=above:$\strut p$] (p) at (spath cs:north 0.35)
            coordinate[label=above:$\strut g$] (g) at (spath cs:north 0.6)
            coordinate[label=above:$\strut k$] (k) at (spath cs:north 0.85)
            coordinate[label=below:$\strut f$] (f) at (spath cs:south 0.25)
            coordinate[label=below:$\strut h$] (h) at (spath cs:south 0.5)
            coordinate[label=below:$\strut q$] (q) at (spath cs:south 0.75)
            ;
          \draw[spath/save=pq] (p) to[out=-90, in=90] (q);
          \draw[spath/save=gf] (g) to[out=-90, in=90] (f);
          \draw[spath/save=kh] (k) to[out=-90, in=90] (h);
          \draw[name intersections={of=pq and gf}] coordinate[dot,label=left:$\alpha$] (alpha) at (intersection-1);
          \draw[name intersections={of=pq and kh}] coordinate[dot,label=right:$\beta$] (beta) at (intersection-1);
        \end{tikzpicture}
      \]

      We now fix natural transformations $\gamma\colon g'\Rightarrow g$ and $\kappa\colon k'\Rightarrow k$ and consider the opcartesian lifts of $\alpha$ and $\beta$ along them respectively:
      \[
        \begin{tikzpicture}[scale=0.5,baseline=($(nw)!0.5!(se)$)]
          \CreateRect{6}{4}
            \path
              coordinate[label=above:$\strut p$] (nw) at (spath cs:north 0.25)
              coordinate[label=above:$\strut g'$] (ne) at (spath cs:north 0.75)
              coordinate[label=below:$\strut f$] (sw) at (spath cs:south 0.25)
              coordinate[label=below:$\strut q$] (se) at (spath cs:south 0.75)
            ;
            \draw[spath/save = pq] (nw) to[out=-90,in=90] (se);
            \draw[spath/save = gf] (ne) to[out=-90,in=90] (sw);
            \path coordinate[dot,label=right:$\gamma$] (gamma) at (spath cs:gf 0.25);
            \path[name intersections={of=pq and gf}]
              coordinate[dot,label=below:$\alpha$] (cell) at (intersection-1)
            ;
        \end{tikzpicture}
        \quad
        \begin{tikzpicture}[scale=0.5,baseline=($(nw)!0.5!(se)$)]
          \CreateRect{6}{4}
            \path
              coordinate[label=above:$\strut q$] (nw) at (spath cs:north 0.25)
              coordinate[label=above:$\strut k'$] (ne) at (spath cs:north 0.75)
              coordinate[label=below:$\strut h$] (sw) at (spath cs:south 0.25)
              coordinate[label=below:$\strut r$] (se) at (spath cs:south 0.75)
            ;
            \draw[spath/save = qr] (nw) to[out=-90,in=90] (se);
            \draw[spath/save = kh] (ne) to[out=-90,in=90] (sw);
            \path coordinate[dot,label=right:$\kappa$] (kappa) at (spath cs:kh 0.25);
            \path[name intersections={of=qr and kh}]
              coordinate[dot,label=below:$\beta$] (cell) at (intersection-1)
            ;
        \end{tikzpicture}
      \]

      If we compute the horizontal composition $\gamma_!\alpha * \kappa_!\beta\colon p \to r$ in $\FIBFib$, we see that we get precisely the opcartesian lift of $\alpha*\beta$ along $\gamma*\kappa\colon g'*k'\to g*k$:
      \[
        \begin{tikzpicture}
          \CreateRect{6}{4}
          \path
            coordinate[label=above:$\strut p$] (p) at (spath cs:north 0.35)
            coordinate[label=above:$\strut g'$] (g) at (spath cs:north 0.6)
            coordinate[label=above:$\strut k'$] (k) at (spath cs:north 0.85)
            coordinate[label=below:$\strut f$] (f) at (spath cs:south 0.25)
            coordinate[label=below:$\strut h$] (h) at (spath cs:south 0.5)
            coordinate[label=below:$\strut q$] (q) at (spath cs:south 0.75)
            ;
          \draw[spath/save=pq] (p) to[out=-90, in=90] (q);
          \draw[spath/save=gf] (g) to[out=-90, in=90] (f);
          \draw[spath/save=kh] (k) to[out=-90, in=90] (h);
          \path coordinate[dot,label=right:$\gamma$] (gamma) at (spath cs:gf 0.25);
          \path coordinate[dot,label=right:$\kappa$] (kappa) at (spath cs:kh 0.25);
          \draw[name intersections={of=pq and gf}] coordinate[dot,label=left:$\alpha$] (alpha) at (intersection-1);
          \draw[name intersections={of=pq and kh}] coordinate[dot,label=right:$\beta$] (beta) at (intersection-1);
        \end{tikzpicture}
        \qedhere
      \]
    \end{proof}

    \begin{cor}\label{cor:2-opfibration}
      The codomain 2-functor $\partial_1\colon \FIBFib\to \CATFib$ is a 2-opfibration (coop-2-fibration) in the sense of Buckley~\cite{buckley:2014}.
    \end{cor}

    \begin{thm}
      The codomain 2-functor $\partial_1\colon \FIBFib\to \CATFib$ is a pseudo-fibration (\ie a fibration of bicategories), with cartesian lifts induced by the oplax base change $\OplDiag{p}\colon \FIB{A}\to \FIB{b}$ for fibrations $p\colon A\rightarrowtriangle B$.
    \end{thm}

    \begin{proof}
      We have seen in Corollary~\ref{cor:2-opfibration} that $\partial_1\colon \FIBFib\to \CATFib$ is a 2-opfibration in the strict sense, with opcartesian lifts induced by the 2-functors $\OplSum{p}\colon \FIB{A}\to \FIB{b}$ for fibrations $p\colon A\rightarrowtriangle B$. If these 2-functors had right 2-adjoints, we would then be able to say that they exhibit $\partial_1\colon \FIBFib\to \CATFib$ as a 2-fibration; because they have only pseudo-adjoints, we have instead only a bicategorical version of the fibration condition.
    \end{proof}

\end{xsect}

\begin{xsect}{Conclusions and future work}
  \begin{xsect}{Relationship to weak bisimulation in concurrency semantics}
    There is one additional piece of related work that deserves disussion, namely the \emph{saturation monad} of Fiore, Cattani, and Winskel~\cite{fiore-cattani-winskel:1999}. In the presheaf semantics of concurrency, a closed process is interpreted as a presheaf on a ``path category'' $\mathbb{P}$---such as the category of pomsets. Joyal, Nielsen, and Winskel~\cite{joyal-nielsen-winskel:1996} showed that \emph{strong bisimulation} of processes has a very simple description in terms of the presheaf semantics:
    \begin{enumerate}
      \item An \emph{open map} $f\colon A\to B$ of presheaves is a natural transformation that has the right lifting property with respect to every $\Yo{\mathbb{P}}\prn{x}\to \Yo{\mathbb{P}}\prn{y}$.
      \item A \emph{strong bisimulation} between processes $A$ and $B$ is a span $A\gets C \to B$ of open maps.
    \end{enumerate}

    \noindent
    Open map bisimulation corresponds exactly to \emph{real} strong bisimulation within the image of various full embeddings of concrete concurrency models into $\hat{\mathbb{P}}$. The result of Fiore, Cattani, and Winskel~\cite{fiore-cattani-winskel:1999} is to reduce \emph{weak} bisimulation to (strong) open map bisimulation using a saturation monad on presheaves, by analogy with existing saturation monads on concrete concurrency models like labelled transition systems.

    \begin{xsect}{The 1-dimensional saturation adjunction}
      For a given hiding functor $h\colon \mathbb{P}\to\mathbb{Q}$, Fiore~\etal's saturation monad on $\hat{\mathbb{P}}$ is obtained from the following 1-categorical adjunction, where we have again written $\mathbf{cat}_s$ for the 1-category of small 1-categories:
      \[
        \begin{tikzcd}
          \hat{\mathbb{P}}
            \arrow[r,hookrightarrow,yshift=1.5ex]
            \arrow[r,phantom,"\bot"]
          &
          \mathbf{cat}_s/\mathbb{P}
            \arrow[r,yshift=1.5ex,"\Sum{h}"]
            \arrow[r,phantom,"\bot"]
            \arrow[l,yshift=-1.5ex,""]
          &
          \mathbf{cat}_s/\mathbb{Q}
            \arrow[l,yshift=-1.5ex,"\Diag{h}"]
        \end{tikzcd}
      \]

      The left adjoint $\hat{\mathbb{P}}\hookrightarrow\mathbf{cat}_s/\mathbb{P}$ depicted above takes a presheaf to the display of its category of elements over $\mathbb{P}$. The composite right adjoint sends a displayed category $X$ over $\mathbb{Q}$ to the presheaf
      $
        \mathbf{cat}_s/\mathbb{Q}\prn{\Sum{h}\prn{\Yo{\mathbb{P}}\prn{-}},X}
      $.
    \end{xsect}

    \begin{xsect}{2-dimensional saturation}
      In order to compare Fiore~\etal's saturation construction on an even playing field with our relative Hofmann--Streicher lifting 2-functor, we must upgrade the former to a pseudo-adjunction. First we recall the 2-adjunction $U_{\mathbb{P}}\dashv N_{\mathbb{P}}\colon\CAT/\mathbb{P}\to\FIB{\mathbb{P}}$ that gives \emph{co-free} fibred categories, having recently been studied by Emmenegger, Mesiti, Rosolini, and Streicher~\cite{emmenegger-mesiti-rosolini-streicher:2024}.

      Then the spirit of Fiore~\etal's saturation adjunction is most faithfully captured by the following pseudo-adjunction:
      \[
        \begin{tikzcd}
          \FIB{\mathbb{P}}
            \arrow[r,hookrightarrow,yshift=1.5ex,"U_{\mathbb{P}}"]
            \arrow[r,phantom,"\bot"]
          &
          \CAT/\mathbb{P}
            \arrow[r,yshift=1.5ex,"\Sum{h}"]
            \arrow[r,phantom,"\bot"]
            \arrow[l,yshift=-1.5ex,"N_{\mathbb{P}}"]
          &
          \CAT/\mathbb{Q}
            \arrow[l,yshift=-1.5ex,"\Diag{h}"]
        \end{tikzcd}
      \]

      A displayed category $X\in\CAT/\mathbb{Q}$ is sent by the right pseudo-adjoint above to the (split) fibred category corresponding under straightening to the following diagram of categories:
      \[
        N_{\mathbb{P}}\prn{\Diag{h}\prn{X}}_\bullet \simeq
        \CAT/\mathbb{Q}\prn{\Sum{h}\prn{\Yo{\mathbb{P}}\prn{-}},X}
      \]
    \end{xsect}

    \begin{xsect}{Comparing saturation with relative Hofmann--Streicher lifting}
      We are now able to compare the 2-dimensional version of the saturation adjunction with our relative Hofmann--Streicher lifting adjunction. We will see that the two constructions agree only on \emph{co-free} fibrations below.
      \begin{obs}
        When $h\colon \mathbb{P}\rightarrowtriangle\mathbb{Q}$ is a fibration, the upgraded right pseudo-adjoint $N_{\mathbb{P}}\circ \Diag{h}\colon \CAT/\mathbb{Q}\to\FIB{\mathbb{P}}$ of Fiore~\etal factors through our own $\OplDiag{h}\colon \FIB{\mathbb{Q}}\to\FIB{\mathbb{P}}$ up to equivalence as follows:
        \[
          \DiagramSquare{
            nw = \CAT/\mathbb{Q},
            ne = \FIB{\mathbb{Q}},
            sw = \CAT/\mathbb{P},
            se = \FIB{\mathbb{P}},
            north = N_{\mathbb{Q}},
            west = \Diag{h},
            east = \OplDiag{h},
            south = N_{\mathbb{P}},
            width = 2.5cm,
          }
        \]
      \end{obs}

      \begin{proof}[Sketch]
        Letting $X\in\CAT/\mathbb{Q}$ be a displayed category, we compute as follows:
        \begin{align*}
          \OplDiag{h}\prn{N_{\mathbb{Q}}\prn{X}}_\bullet
          &\simeq
          \FIB{\mathbb{P}}\prn{\Yo{\mathbb{P}}\prn{-},\OplDiag{h}\prn{N_{\mathbb{Q}}\prn{X}}}
          \\
          &\simeq
          \FIB{\mathbb{Q}}\prn{\OplSum{h}\prn{\Yo{\mathbb{P}}\prn{-}},N_{\mathbb{Q}}\prn{X}}
          \\
          &\simeq
          \CAT/\mathbb{Q}\prn{U_{\mathbb{Q}}\prn{\OplSum{h}\prn{\Yo{\mathbb{P}}\prn{-}}},X}
          \\
          &\simeq
          \CAT/\mathbb{Q}\prn{\Sum{h}\prn{\Yo{\mathbb{P}}\prn{-}},X}
          \\
          &\simeq
          N_{\mathbb{P}}\prn{\Diag{h}\prn{X}}_\bullet
          \qedhere
        \end{align*}
      \end{proof}

      \begin{obs}
        The co-free fibration on a displayed category $X\in\CAT/\mathbf{1}$ over the point is precisely $X$, \ie we have $N_{\mathbf{1}}\prn{X}\simeq X$.
      \end{obs}

      From these two observations, we immediately obtain the following.

      \begin{cor}
        Saturation agrees with (relative) Hofmann--Streicher lifting over the point. More precisely, for any categories $C$ and $E$, we have $N_{\mathbb{P}}\prn{\Diag{!_C}\prn{E}} \simeq \OplDiag{!_C}\prn{E}$.
      \end{cor}

      \noindent
      Therefore, our construction and that of Fiore~\etal are two \emph{different} generalisations of Hofmann--Streicher lifting; ours is the correct one for internalising and iterating Hofmann--Streicher lifting, whereas theirs is the correct one for reducing weak bisimulation to (strong) open map bisimulation in presheaf models of concurrency. It would be very interesting indeed to find out whether our own relative Hofmann--Streicher lifting does anything useful in the world of concurrency semantics.
    \end{xsect}
  \end{xsect}

  \begin{xsect}{Acknowledgements}
    We are very grateful to Mathieu Anel, Nathanael Arkor, Steve Awodey, Reid Barton, Marcelo Fiore, and Daniel Gratzer for useful conversations during the course of this project. Steve's clear explanations of his functorial version of Hofmann--Streicher lifting were extremely helpful, and it was a lecture of his in Stockholm that inspired the second named author to pursue the present work. Finally, we thank Thomas Streicher for his lifetime of contribution to type theory and category theory, and for his kindness over the years. We are deeply saddened that we will no longer be able to discuss this work with him.
  \end{xsect}
\end{xsect}

\nocite{awodey:2024:universes}
\bibliographystyle{alphaurl}
\bibliography{refs-bibtex}

\end{document}